\documentclass[10pt,a4paper]{amsart}

\usepackage{amssymb}
\usepackage[utf8]{inputenc}
\usepackage[english]{babel}
\usepackage{amsfonts}
\usepackage{amsmath}
\usepackage{comment}
\DeclareFontFamily{U}{mathx}{\hyphenchar\font45}
\DeclareFontShape{U}{mathx}{m}{n}{
      <5> <6> <7> <8> <9> <10>
      <10.95> <12> <14.4> <17.28> <20.74> <24.88>
      mathx10
      }{}
\DeclareSymbolFont{mathx}{U}{mathx}{m}{n}
\DeclareFontSubstitution{U}{mathx}{m}{n}
\DeclareMathAccent{\widecheck}{0}{mathx}{"71}
\DeclareMathAccent{\wideparen}{0}{mathx}{"75}

\usepackage{amsthm}
\usepackage{graphicx}
\usepackage{xcolor}
\usepackage{color}
\usepackage{enumitem}
\usepackage[bookmarksopen,bookmarksdepth=3,colorlinks,citecolor=red,pagebackref,hypertexnames=false]{hyperref}
\usepackage[nameinlink]{cleveref}

\usepackage{dsfont}

\newtheorem{main-theorem}{Theorem}
\newtheorem{proposition}{Proposition}[section]
\newtheorem{theorem}[proposition]{Theorem}

\newtheorem{corollary}[proposition]{Corollary}
\newtheorem{lemma}[proposition]{Lemma}
\theoremstyle{remark}
\newtheorem{remark}[proposition]{Remark}

\theoremstyle{definition}

\newtheorem*{acknowledgements}{Acknowledgements}

\DeclareMathOperator{\supp}{supp}

\newcommand{\R}{\mathbb{R}}
\newcommand{\C}{\mathbb{C}}

\newcommand{\N}{\mathbb{N}}

\newcommand{\dd}{\mathrm{d}}
\newcommand{\tm}{\mathrm{t}}
\newcommand{\x}{\mathrm{x}}
\newcommand{\Id}{\mathrm{Id}}
\newcommand{\Orth}{\mathrm{O}}

\newcommand{\marnote}[1]{\marginpar{\footnotesize #1}} 

\title[Initial-to-final for the heat operator]{The initial-to-final inverse problem\\
for the heat operator}

\author[Alberola]{Enric Alberola}
\address[Enric Alberola]{Basque Center for Applied Mathematics, Bilbao, Spain}
\email{\href{mailto:ealberola@bcamath.org}{\textrm{ealberola@bcamath.org}}}

\author[Aroua]{Nesrine Aroua}
\address[Nesrine Aroua]{MaLGa Center, Department of Mathematics, University of Genoa, Via Dodecaneso 35, 16146 Genova, Italy.
}
\email{\href{mailto:nesrine.aroua@edu.unige.it}{\textrm{nesrine.aroua@edu.unige.it}}}

\author[Caro]{Pedro Caro}
\address[Pedro Caro]{
Basque Center for Applied Mathematics and Ikerbasque (Basque Foundation for Science) Bilbao, Spain}
\email{\href{mailto:pcaro@bcamath.org}{\textrm{pcaro@bcamath.org}}}

\begin{document}

\begin{abstract}
We study an inverse problem for the heat equation in a medium that generates internal heat at a 
rate proportional to the heat density. The goal consists of determining the
heat-generation coefficient 
$V$ from the knowledge of the initial-to-final map, which assigns to every initial heat density its 
final heat density. We state and prove a uniqueness result for the heat operator in
a region modelled by $\R^n$ with $n \geq 2$.
We assume that $V$ is bounded and decays super-exponentially.
Our approach relies on constructing exponentially-growing solutions for
the corresponding heat operator. A key contribution of our approach is to provide a 
weighted $L^2$-estimate, which follows from the moment generating function of a Gaussian random 
variable. This extends the initial-to-final-state inverse problem, previously studied for the 
Schrödinger equation, to the parabolic setting.
\end{abstract}

\date{\today}


\maketitle

\section{Introduction}\label{sec: intro}
Let $\R^n$ with $n \in \N$ represent a region where thermal conduction takes place.
If we assume that this region generates its own heat at a rate given by 
$V(t, x) u(t, x)$, at each time $t \in (0, T)$ and at every point $x \in \R^n$,
then the heat density 
$u : (t, x) \in [0,T] \times \R^n \mapsto u(t, x) $,
initially prescribed by $f : x \in \R^n \mapsto f(x)$, satisfies the following initial-value problem
\begin{equation}
	\left\{
		\begin{aligned}
		& \partial_\tm u - \Delta u = V u & & \textnormal{in} \enspace (0,T) \times \R^n, \\
		& u(0, \centerdot) = f &  & \textnormal{in} \enspace \R^n.
		\end{aligned}
	\right.
	\label{pb:IVP}
\end{equation}
Given a \textit{heat-generation} coefficient $V \in L^1((0,T); L^\infty(\R^n))$
and an initial condition $f \in L^2(\R^n)$,
there exists a unique solution
$u \in C([0,T]; L^2(\R^n))$.
Additionally, there exists a constant $C$ that depends on $V$ and $T$ such that
\[ \sup_{t \in [0,T]} \| u(t, \centerdot) \|_{L^2 (\R^n)} \leq C \| f \|_{L^2(\R^n)}. \]
This allows us to define a bounded linear map given by
\[ \mathcal{U}_T : f \in L^2(\R^n) \longmapsto u(T, \centerdot) \in L^2(\R^n), \]
that we call the initial-to-final map.

We seek conditions on $V$ that ensure its determination from $\mathcal{U}_T$. This problem was 
first formulated by Caro and Ruiz \cite{zbMATH07801151} in the context of the Schrödinger equation
\[ i \partial_\tm u = -\Delta u + V u \]
as a theoretical framework for data-driven prediction in quantum mechanics.
The motivation stems from the following data-driven paradigm: if one knows the final state of a 
quantum system for every possible initial state, can one determine the underlying Hamiltonian 
governing the evolution? In the exact formulation, this amounts to recovering the potential $V$ 
from the initial-to-final-state map $\mathcal{U}_T$.

For the Schrödinger equation, Caro and Ruiz \cite{zbMATH07801151} proved unique
determination of a potential
$V \in L^1((0,T); L^\infty(\R^n))$ that has \emph{super-exponential decay}, that is,
$ e^{\rho |\x|} V \in L^\infty(\Sigma) $ for all $\rho > 0$, where 
$\Sigma = (0, T) \times \R^n$.
Subsequent work has extended this result to unbounded potentials \cite{arXiv:2512.04796}, 
to time-independent potentials with super-linear decay \cite{zbMATH08122191}, and to critically 
singular potentials \cite{arXiv:2602.12122}. These developments demonstrate the robustness of the 
initial-to-final-state inverse problem and its relevance to quantum mechanics.

In this paper, we initiate the study of the initial-to-final inverse problem for the heat equation. 
While the general structure of the problem parallels that of the Schr\"odinger case, the parabolic 
nature of the heat operator introduces significant differences and opportunities to understand
the problem from a different point of view. Our main result establishes uniqueness for the heat 
equation in dimension $n \geq 2$ under the same super-exponential decay assumption as in 
\cite{zbMATH07801151}.

\begin{main-theorem}\label{th:uniqueness} \sl
Consider $V_1$ and $V_2$ in $L^1((0, T); L^\infty (\R^n))$ with $n \geq 2$.
Let $\mathcal{U}_T^j$ with $j \in \{ 1, 2 \}$ denote 
the initial-to-final map associated to $V_j$. If 
$V_1$ and $V_2$ have super-exponential decay, then
\[ \mathcal{U}_T^1 = \mathcal{U}_T^2 \Rightarrow V_1 = V_2. \]
\end{main-theorem}

Our approach to proving \Cref{th:uniqueness} follows the general strategy developed in
\cite{zbMATH07801151} adapted to the parabolic setting.
As we will see in \Cref{prop:integral_formula}, the identity 
$\mathcal{U}_T^1 = \mathcal{U}_T^2$ implies that
\begin{equation}\label{id:orthogonality}
\int_\Sigma (V_1 - V_2) u_1 v_2 \, = 0
\end{equation}
for all $u_1, v_2 \in C([0,T]; L^2(\R^n))$ solutions of the equations 
$(\partial_\tm - \Delta - V_1) u_1 = 0$ and $(\partial_\tm + \Delta + V_2) v_2 = 0$ in 
$\Sigma$.

In order to motivate our approach, let us 
consider the following toy model: Assume that 
$F \in L^1(\R \times \R^n)$ with support in $[0, T] \times \R^n$ satisfies that
\[ \int_\Sigma F  [e^{\tm \Delta}f] [e^{(T - \tm) \Delta}g] \, = 0 \quad \forall f, g \in \mathcal{S}(\R^n). \]
Then, can one ensure that $F(t, x) = 0$ for a.e.
$(t, x) \in \R \times \R^n$? Let us see that the answer is affirmative.
Start by noting that
\[ [e^{\tm \Delta}f] (t,x) = \frac{1}{(2\pi)^{n/2}} \int_{\R^n} e^{ix \cdot \xi} e^{-t|\xi|^2} \widehat{f}(\xi) \, \dd \xi \quad \forall (t, x) \in (0, \infty) \times \R^n \]
solves the equation $(\partial_\tm - \Delta) u = 0$ in $(0, \infty) \times \R^n$, while
\[ [e^{(T - \tm) \Delta}g] (t,x) = \frac{1}{(2\pi)^{n/2}} \int_{\R^n} e^{ix \cdot \eta} e^{-(T-t)|\eta|^2} \widehat{g}(\eta) \, \dd \eta \quad \forall (t, x) \in (-\infty, T) \times \R^n \]
solves the equation $(\partial_\tm + \Delta) v = 0$ in $(-\infty, T) \times \R^n$. The 
assumption on $F$ implies that
\[ \int_{\R \times \R^n} e^{t(|\eta|^2 - |\xi|^2)} e^{i x \cdot (\xi + \eta)} F(t, x) \, \dd (t, x) = 0 \quad \forall \xi, \eta \in \R^n. \]
For every $(\tau, \kappa) \in \R \times \R^n$ such that $\kappa \neq 0$, we choose
\begin{align*}
    \eta=- \frac{1}{2}\big(1+\frac{\tau}{|\kappa|^2}\big)\kappa   ~~~\text{and}~~~\xi= -\frac{1}{2}\big(1-\frac{\tau}{|\kappa|^2}\big)\kappa.
\end{align*}
Note that $\xi+\eta = - \kappa$ and $|\eta|^2 - |\xi|^2 = \tau$.
Then, for every $\kappa \in \R^n \setminus \{ 0 \}$ we have that
\begin{equation}
\label{id:vanishing_LF}
\int_\R e^{t\tau} G_\kappa (t) \,\dd t = 0 \quad \forall \tau \in \R,
\end{equation}
with $ G_\kappa (t) = \mathcal{F}[F(t, \centerdot)](\kappa) $---the Fourier transform of $F$
with respect to the spatial variable.
Since $\supp F \subset [0,T] \times \R^n$, we have that the Fourier transform of $G_\kappa$ 
can be analytically extended to $\C$. The condition 
\eqref{id:vanishing_LF} implies that this extension, denoted by $\widehat{G_\kappa}$, 
satisfies that $\widehat{G_\kappa}(i\tau) = 0$ for all $\tau \in \R$. Then, by analytic 
continuation $\widehat{G_\kappa}(\zeta) = 0$ for all $\zeta \in \C$. In particular,
$\widehat{G_\kappa}(\sigma) = 0$ for all $\sigma \in \R$. Since 
$\widehat{G_\kappa}(\sigma) = \widehat{F}(\sigma, \kappa)$---the full Fourier transform of 
$F$, we have that $\widehat{F}(\sigma, \kappa) = 0$ for every
$(\tau, \kappa) \in \R \times \R^n$ such that $\kappa \neq 0$. The assumption 
$F \in L^1(\R \times \R^n)$ makes its Fourier transform continuous, and by continuity, we 
know that $\widehat{F}(\sigma, \kappa) = 0$ for all $(\tau, \kappa) \in \R \times \R^n$.
Finally, by the injectivity of the Fourier transform, we know that $F(t, x) = 0$ for 
a.e. $(t, x) \in \R \times \R^n$.

After this observation for the toy model, it seems natural to address the real problem
by constructing solutions $u$ and $v$ that are perturbations of $e^{\tm \Delta}f$ and 
$e^{(T - \tm) \Delta}g$ respectively. The correction terms introduced should vanish in some 
appropriate sense so that we can recover the argument from the toy model. Adapting the ideas in 
\cite{zbMATH07801151} to the parabolic case, we will construct solutions
\begin{align*}
& u_1^\nu = e^{|\nu|^2 \tm + \nu \cdot \x} (u_1^\sharp + u_1^\flat)\\
& v_2^\nu = e^{|\nu|^2 (T - \tm) - \nu \cdot \x} (v_2^\sharp + v_2^\flat)
\end{align*}
for any $\nu \in \R^n $ sufficiently large. The leading terms $u_1^\sharp$ and $v_2^\sharp$ are
specifically chosen to be constant in the direction of $\nu$ and solutions of the heat equation in 
the hyperplane orthogonal to $\nu$. These choices are motivated by the fact that it makes the 
correction terms $u_1^\flat$ and $v_2^\flat$ negligible as $|\nu| $ tends to $ \infty$.

The construction of the correction terms is one of the difficulties overcome in this paper since we 
had to identify the appropriate estimate. This task is carried out in \Cref{sec:CGO}, while the
crucial estimate is proved in \Cref{sec:boundedness} as a consequence of 
the moment generating function of a Gaussian random variable, see \Cref{lem:heat-exponential}.

Once the exponentially-growing solutions $u_1^\nu$ and $v_2^\nu$ are available,
the second challenge is to extend the identity \eqref{id:orthogonality} from solutions 
$u_1, v_2 \in C([0,T]; L^2(\R^n))$ to $u_1^\nu$ and $v_2^\nu$ as above.
This step is more straightforward than for the Schrödinger equation since we can exploit the 
parabolic regularity of the heat equation. This is achieved in \Cref{prop:extension-ortho}.

The inverse problem of determining coefficients in parabolic equations from boundary measurements has been extensively studied in the literature. Canuto and Kavian \cite{zbMATH01578850} addressed the identifiability of the density, diffusion, and heat-generation coefficients in a heat equation from boundary measurements on two portions of the boundary, proving uniqueness of the boundary spectral data under suitable hypotheses. In the context of convection-diffusion equations, Bellassoued and Rassas \cite{zbMATH07173395} established log-type stability estimates for the unique determination of the first- and zeroth-order coefficients from the Dirichlet-to-Neumann map in dimension $n \geq 3$. Subsequently, Bellassoued and Ben Fraj \cite{zbMATH07335388} extended these results to the partial data setting, proving logarithmic stability estimates for time-dependent convection and heat-generation coefficients from measurements on an arbitrary open subset of the boundary.

The determination of convection terms in diffusion equations was further investigated
by Caro and Kian \cite{arXiv:1812.08495}, who proved the unique recovery of a general vector-valued
first-order coefficient depending on both time and space variables
from lateral boundary measurements. Choulli and Kian \cite{zbMATH06869653} established logarithmic stability in determining the time-dependent zero-order coefficient from a partial Dirichlet-to-Neumann map, introducing a parabolic Carleman inequality that played a crucial role in constructing CGO solutions vanishing on parts of the boundary.

More recently, Sahoo and Vashisth \cite{zbMATH07170240} studied the unique determination of the convection term and time-dependent density coefficient from partial boundary measurements, while Senapati and Vashisth \cite{zbMATH07597186} established log-log and log-log-log stability estimates for the same problem. Purohit \cite{zbMATH08020896} employed nonlinear Carleman weights to recover time-dependent convection and density terms from measurements on a small subset of the boundary. Mishra, Purohit, and Vashisth \cite{zbMATH08109740} extended the partial data inverse problem for the time-dependent convection-diffusion equation to admissible Riemannian manifolds, recovering both first- and zeroth-order time-dependent perturbations.

Feizmohammadi \cite{zbMATH07802400} studied an inverse boundary value problem for isotropic nonautonomous heat flows, proving uniqueness in all dimensions under an assumption on the thermal diffusivity related to the construction of exponential solutions. Kian \cite{zbMATH07543698} considered the simultaneous determination of multiple classes of parameters---including convection coefficients, density, internal sources, and fractional order---for diffusion equations from a single boundary measurement. Rassas \cite{zbMATH07829439} obtained H\"older stability estimates for determining time-dependent scalar and vector coefficients in the convection-diffusion equation from the Cauchy data set.

More recently, Buisson \cite{arXiv:2607.13778} addressed the inverse problem of determining zeroth- and first-order coefficients in a parabolic equation from partial boundary measurements, developing a new class of Carleman estimates with nonlinear weights and establishing the parabolic analogue of limiting Carleman weights in the elliptic setting. Unlike the existing literature on partial data inverse problems for parabolic equations, both the phase of the special solutions and the weight of the corresponding Carleman estimates are nonlinear, representing a significant advancement in the field.

In contrast to these works, which focus on inverse problems for convection-diffusion equations with boundary measurements, the present paper addresses the initial-to-final inverse problem for the heat equation, where data consist of the final heat density of the system for every initial heat density, and the zeroth-order perturbation is recovered from the initial-to-final map rather than from boundary measurements. This formulation allows us to consider non-compact perturbations, which seem beyond the scope of standard boundary data approaches.

\subsection*{Contents} Here we summarize the contents of the paper. The orthogonality relation
\eqref{id:orthogonality} is derived from the result on \Cref{sec:int_formula}. The details 
about how to construct the exponentially-growing solutions are given in \Cref{sec:CGO}.
In \Cref{sec:proof_main-th}, we prove \Cref{th:uniqueness}. \Cref{sec:boundedness} is
devoted to prove the key inequality used in the construction of the correction terms.
\Cref{sec:extension_ortho} contains the technical details to extend the orthogonality relation
to exponentially-growing solutions. We end the paper by including \Cref{app:integration-parts}
where we prove an integral formula that is fundamental for our arguments.

\section{Initial-to-final integral formula}\label{sec:int_formula}
\begin{proposition}\label{prop:integral_formula}\sl Consider 
$V_1, V_2 \in L^1((0,T); L^\infty(\R^n)) $ and let $ \mathcal{U}_T^1$ and 
$ \mathcal{U}_T^2$ denote their corresponding initial-to-final maps.
Then,
\[\int_\Sigma (V_1 - V_2) u_1 v_2 \, = \int_{\R^n} (\mathcal{U}_T^1 - \mathcal{U}_T^2)f \, g\]
with $u_1 \in C([0,T]; L^2(\R^n))$ the solution of the initial-value problem \eqref{pb:IVP}
with coefficient $V_1$, and $v_2 \in C([0,T]; L^2(\R^n))$ the solution of the following 
final-value problem
\begin{equation}
	\left\{
		\begin{aligned}
		& (\partial_\tm + \Delta + V_2) v_2 = 0 & & \textnormal{in} \enspace \Sigma, \\
		& v_2(T, \centerdot) = g &  & \textnormal{in} \enspace \R^n.
		\end{aligned}
	\right.
	\label{pb:FVP}
\end{equation}
\end{proposition}

\begin{proof}
Beside the solutions $u_1$ and $v_2$, let us consider $u_2 \in C([0,T]; L^2(\R^n))$ the 
solution of the initial-value problem \eqref{pb:IVP} with coefficient $V_2$. After an
integration by parts as the one in \Cref{L_A.1} one can see that
\[\int_\Sigma (V_j - V_2) u_j v_2 = \int_{\R^n} [u_j(T, \centerdot) v_2 (T, \centerdot) - u_j(0, \centerdot) v_2 (0, \centerdot)],\]
which implies that
\[\int_\Sigma (V_1 - V_2) u_1 v_2 = \int_{\R^n} [u_1(T, \centerdot) v_2 (T, \centerdot) - u_1(0, \centerdot) v_2 (0, \centerdot)],\]
and
\[\int_{\R^n} [u_2(T, \centerdot) v_2 (T, \centerdot) - u_2(0, \centerdot) v_2 (0, \centerdot)] \, = 0. \]
Since $u_1(0, \centerdot) = u_2(0, \centerdot) = f$, we conclude that
\[\int_\Sigma (V_1 - V_2) u_1 v_2 = \int_{\R^n} [u_1(T, \centerdot) - u_2(T, \centerdot)] v_2 (T, \centerdot).\]
Rewriting the identity with $u_j(T, \centerdot) = \mathcal{U}_T^j f$ and 
$v_2(T, \centerdot) = g$, we obtained the formula in the statement.
\end{proof}

\section{Exponentially-growing solutions}\label{sec:CGO}
Given a suitable coefficient $V$ defined in $\Sigma$,
we look for solutions of
\begin{equation}
(\partial_\tm - \Delta - V) u = 0 \enspace \textnormal{in} \enspace \Sigma,
\label{eq:heat}
\end{equation}
in the form
\begin{equation}
\label{id:GO}
u = e^{|\nu|^2 \tm + \nu \cdot \x} (u^\sharp + u^\flat),
\end{equation}
with $\nu \in \R^n \setminus \{ 0 \}$, and
$u^\sharp$ chosen so that $e^{|\nu|^2 \tm + \nu \cdot \x} u^\sharp$ is a solution of 
$(\partial_\tm - \Delta) (e^{|\nu|^2 \tm + \nu \cdot \x} u^\sharp) = 0$ in $\Sigma$,
or equivalently $(\partial_\tm - \Delta - 2 \nu \cdot \nabla) u^\sharp = 0 $ in
$\Sigma$. This choice forces $u^\flat$ to satisfy
\begin{equation}
(\partial_\tm - \Delta - 2 \nu \cdot \nabla - V) u^\flat = V u^\sharp \enspace \textnormal{in} \enspace \Sigma.
\label{eq:remainder-equation}
\end{equation}

\subsection{The leading term \texorpdfstring{$u^\sharp$}{usharp}}
We choose $u^\sharp$ so that $\nu \cdot \nabla u^\sharp = 0$ and $(\partial_\tm - \Delta) u^\sharp = 0$ in $\Sigma$. For that, it is convenient to introduce some notation.

For $\hat{\nu} = \nu / |\nu|$, consider the hyperplane
$ H_{\hat{\nu}} = \{ x \in \R^n : \hat{\nu} \cdot x = 0 \} $
endowed with the measure $\sigma_{\hat{\nu}}$.\footnote{The measure $\sigma_{\hat{\nu}}$ is
defined as the push-forward of the 
Lebesgue measure in $\R^{n-1}$ via the map
\[ y^\prime \in \R^{n-1} \longmapsto y=(y^\prime, 0) \in \R^{n-1} \times \{ 0 \} \longmapsto Q y \in H_{\hat{\nu}}, \]
where $Q$ is any matrix in $\Orth(n)$ such that $\hat{\nu} = Q e_n$
with $\{ e_1, \dots , e_n \}$ the standard basis of $\R^n$. The definition of 
$\sigma_{\hat{\nu}}$ is independent of the choice of $Q$.}

For $\nu \in \R^n \setminus \{0\}$ and $\psi \in \mathcal{S} (\R^n)$, we choose
\begin{equation}
\label{id:ushrap}
u^\sharp (t, x) = \frac{1}{(2\pi)^\frac{n-1}{2}} \int_{H_{\hat{\nu}}} e^{ix \cdot \xi} e^{- t |\xi|^2} {\psi}(\xi) \, \dd \sigma_{\hat{\nu}} (\xi) \quad \forall (t, x) \in \Sigma.
\end{equation}
Note that $u^\sharp \in C([0,T]; L^\infty(\R^n))$ and
\[\sup_{t \in [0, T]} \| u^\sharp (t, \centerdot) \|_{L^\infty(\R^n)} \leq \| \psi \|_{L^1(H_{\hat \nu})}.\]

Another useful estimate for $u^\sharp$ is: for every 
$\varepsilon > 0$ we have that
\begin{equation}\label{in:usharp_full-space}
\| e^{-\varepsilon |\hat \nu \cdot x|} u^\sharp \|_{C([0, T]; L^2(\R^n))} \leq \| e^{-\varepsilon |\centerdot|} \|_{L^2(\R)} \| u^\sharp \|_{C([0, T]; L^2(H_{\hat \nu}))} \lesssim \varepsilon^{-1/2} \| \psi \|_{L^2(H_{\hat \nu})}.
\end{equation}
The implicit constant in the last inequality is absolute.

\subsection{The correction term \texorpdfstring{$u^\flat$}{uflat}}
Start with a simple observation: $v$ is solution of
$(\partial_\tm - \Delta) v = \psi$ in $\Sigma$ if and only if
$w (t, x) = v (t, x + 2t \nu)$ is solution of $(\partial_\tm - \Delta - 2 \nu \cdot \nabla) w = \phi $ in
$\Sigma$ with $\phi (t, x) = \psi (t, x + 2t \nu)$.
It is well known that
\[ v: (t, x) \in \overline{\Sigma} \longmapsto \int_0^t \Big( \int_{\R^n} H_n (t-s, x-y) \psi(s, y) \, \dd y \Big) \, \dd s, \]
with $H_n$ denoting the $n$-dimensional heat kernel
\begin{equation}
\label{id:heat_kernel}
H_n: (t, x) \in \Sigma \longmapsto \frac{1}{(4\pi t)^{n/2}} e^{-\frac{|x|^2}{4t}},
\end{equation}
solves $(\partial_\tm - \Delta) v = \psi$ in $\Sigma$ and satisfies
$v(0, \centerdot) = 0$. Then,
\[ w: (t, x) \in \overline{\Sigma} \longmapsto \int_0^t \Big( \int_{\R^n} H_n (t-s, x - y + 2(t-s) \nu) \phi(s, y) \, \dd y \Big) \, \dd s, \]
solves $(\partial_\tm - \Delta - 2 \nu \cdot \nabla) w = \phi $ in $\Sigma$ 
and satisfies $w(0, \centerdot) = 0$. Furthermore,
\[ \|w(t, \centerdot)\|_{L^2(\R^n)} \leq \int_0^t \| H_n (t-s, \centerdot + 2(t-s) \nu) \|_{L^1(\R^n)} \| \phi(s, \centerdot)\|_{L^2(\R^n)} \, \dd s,\]
which implies that
\[ \sup_{t \in [0, T]} \|w(t, \centerdot)\|_{L^2(\R^n)} \leq \| \phi\|_{L^1((0,T); L^2(\R^n))}, \]
since
\begin{equation}
\label{id:heat_L1_tranlation-invariant}
\| H_n (t-s, \centerdot + 2(t-s) \nu) \|_{L^1(\R^n)} = \| H_n (t-s, \centerdot) \|_{L^1(\R^n)} = 1.
\end{equation}

This observation motivates the definition of the operator $S_\nu$ for $\phi $
in the space $ L^1((0,T); L^2(\R^n))$ as
\[ S_\nu \phi: (t, x) \in [0, T] \times \R^n \longmapsto \int_0^t \Big( \int_{\R^n} H_n (t-s, x - y + 2(t-s) \nu) \phi(s, y) \, \dd y \Big) \, \dd s, \]
and shows that it is bounded from $ L^1((0,T); L^2(\R^n))$ to $C([0, T]; L^2(\R^n))$
\begin{equation}
\label{in:Snu_trivial}
\sup_{t \in [0, T]} \|S_\nu \phi (t, \centerdot)\|_{L^2(\R^n)} \leq \| \phi\|_{L^1((0,T); L^2(\R^n))}.
\end{equation}
Additionally, note that $S_\nu \phi (0, \centerdot)= 0$.
Therefore, if we want to solve \eqref{eq:remainder-equation} it is enough to solve
\begin{equation}
\label{eq:Neumann}
(\Id - S_\nu \circ M_V) u^\flat = S_\nu (V u^\sharp),
\end{equation}
with $M_V $ denoting the operator $ \phi \mapsto V \phi$; since applying the differential 
operator $\partial_\tm - \Delta - 2 \nu \cdot \nabla$ to both sides on 
\eqref{eq:Neumann} we would see that $u^\flat$ satisfies the identity 
\eqref{eq:remainder-equation}. A way to solve \eqref{eq:Neumann} consists in finding 
a Banach space where the operator $S_\nu \circ M_V$ is a contraction. The inequality 
\eqref{in:Snu_trivial} could be used for that purpose when 
$\| V \|_{L^1((0,T); L^\infty(\R^n))} < 1$. Otherwise, we could first use it 
for a small period of time $[0, \delta]$ with $\delta > 0$ only depending on $V$,
and then for each time interval
$[j \delta, (j+1)\delta]$ with $j \in \N$. Unfortunately, none of these constructions of $u^\flat$
would make it negligible as $|\nu|$ tends to $\infty$. To do so, we need to kill the translation
invariance in \eqref{id:heat_L1_tranlation-invariant} of the norm with respect to the space variable so 
that we capture the translation $2(t-s)\nu$ in the kernel of the operator $S_\nu$. In
\Cref{sec:boundedness}, we will prove that if $\varepsilon > 0$ and $\nu \in \R^n \setminus\{ 0 \}$,
then
\begin{equation}
\label{in:Snu}
\| e^{\varepsilon \hat{\nu} \cdot \x} S_\nu \phi \|_{L^2(\Sigma)} \leq \frac{e^{\varepsilon^2 T}}{2 \varepsilon |\nu|} \| e^{\varepsilon \hat{\nu} \cdot \x} \phi \|_{L^2(\Sigma)}
\end{equation}
for all $\phi \in \mathcal{D}(\Sigma)$.
This inequality already exhibits a decay in $|\nu|$ that is captured by the fact that the 
weight $e^{\varepsilon \hat{\nu} \cdot \x}$ kills the translation invariance in the direction 
of $\nu$---see \Cref{lem:heat-exponential}.
Furthermore, it is trivial that
\begin{equation}
\label{in:MV}
\| e^{\varepsilon \hat{\nu} \cdot \x} M_V \phi \|_{L^2(\Sigma)} \leq \| V \|_{L^\infty ((0, T) \times \R^n)} \| e^{\varepsilon \hat{\nu} \cdot \x} \phi \|_{L^2(\Sigma)}
\end{equation}
for all $\phi \in \mathcal{D}(\Sigma)$.
From the inequalities \eqref{in:Snu} and \eqref{in:MV}, we will derive the existence of 
$u^\flat$ in a suitable Banach space dictated by the previous inequalities. This space will be defined as follows. Let $X_{\varepsilon\hat{\nu}}(\Sigma)$ denote 
the set of equivalence classes $\phi$ of measurable functions 
that are equal almost everywhere in $\Sigma$ so that 
$e^{\varepsilon \hat{\nu} \cdot \x} \phi \in L^2(\Sigma)$.
This set endowed with the norm
\[ \| \phi \|_{X_{\varepsilon \hat{\nu}}(\Sigma)} = \| e^{\varepsilon \hat{\nu} \cdot \x} \phi \|_{L^2(\Sigma)} \]
becomes a Banach space, and $\mathcal{D}(\Sigma)$ can be embedded as a dense 
subspace.

\begin{proposition}\label{Prop CGO1}\sl
Consider $V \in L^\infty (\Sigma) $ such that
\begin{equation*}
\int_{\Sigma} | e^{\varepsilon |x|} V(t, x) |^2 \, \dd (t, x) < \infty
\end{equation*}
for some $\varepsilon > 0$, and set
$c_V = e^{\varepsilon^2 T} \| V \|_{L^\infty (\Sigma)}/(2\varepsilon)$.
Then, for every $\nu \in \R^n$ such that $|\nu| > c_V$,
there exists $u^\flat \in X_{\varepsilon \hat{\nu}}(\Sigma)$ 
that solves \eqref{eq:Neumann}, and satisfies
\begin{equation}
\label{eq:uflat}
u^\flat = S_\nu [V(u^\sharp + u^\flat)].
\end{equation} 
Furthermore, 
\[ \| u^\flat \|_{X_{\varepsilon \hat{\nu}}(\Sigma)}
\leq \frac{e^{\varepsilon^2 T}}{ \varepsilon |\nu|} 
\| V \|_{X_{\varepsilon \hat{\nu}}(\Sigma)}
\| u^\sharp \|_{L^\infty (\Sigma)} \]
for all $|\nu| \geq 2 c_V$.
\end{proposition}

\begin{proof}
    We start by proving that the operator $\Id - S_\nu \circ M_V$
    has a bounded inverse in $X_{\varepsilon \hat{\nu}} (\Sigma)$ for all $|\nu| > c_V$.
By the inequalities \eqref{in:Snu} and \eqref{in:MV} we obtain
    \begin{equation*}
      \| (S_\nu \circ M_V) \phi\|_{X_{\varepsilon \hat{\nu}}(\Sigma)} \leq \frac{e^{\varepsilon^2 T}}{ 2 \varepsilon |\nu|} \|  M_V \phi\|_{X_{\varepsilon \hat{\nu}}(\Sigma)} 
      \leq \frac{c_V}{  |\nu|}  \| \phi\|_{X_{\varepsilon \hat{\nu}}(\Sigma)}.
    \end{equation*}
   Thus, whenever $|\nu|> c_V$, we have that the operator 
       \begin{align*}
   S_\nu \circ M_V:X_{\varepsilon \hat{\nu}}(\Sigma) \longrightarrow   X_{\varepsilon \hat{\nu}}(\Sigma), 
    \end{align*}
    is a contraction. This implies that, whenever $|\nu|> c_V$, $\Id - S_\nu \circ M_V$ has an inverse, and if denoted by $(\Id - S_\nu \circ M_V)^{-1}$, it satisfies
    \begin{align*}
       \| (\Id - S_\nu \circ M_V)^{-1}\|_{\mathcal{L}(X_{\varepsilon \hat{\nu}}(\Sigma))} & \leq \sum_{k=0}^\infty \|S_{\nu} \circ M_{V}\|_{\mathcal{L}(X_{\varepsilon \hat{\nu}}(\Sigma))}^k \\ 
       & = \frac{1}{1-\|S_{\nu} \circ M_{V}\|_{\mathcal{L}(X_{\varepsilon \hat{\nu}}(\Sigma))}} \leq \frac{1}{1-\frac{c_V}{  |\nu|}}.
    \end{align*}
    Here we used the Neumann's series theorem. Therefore, choosing
    $ u^\flat = (\Id - S_\nu \circ M_V)^{-1} \circ S_\nu (V u^\sharp) $
     we obtain
       \begin{align*}
 \| u^\flat \|_{X_{\varepsilon \hat{\nu}}(\Sigma)}
 &\leq \frac{1}{1-\frac{c_V}{  |\nu|}} \| S_\nu (V u^\sharp)\|_{X_{\varepsilon \hat{\nu}}(\Sigma)} \leq \frac{1}{1-\frac{c_V}{  |\nu|}} \frac{e^{\varepsilon^2 T}}{  2 \varepsilon |\nu|} \| V u^\sharp\|_{X_{\varepsilon \hat{\nu}}(\Sigma)}\\
   &\leq \frac{e^{\varepsilon^2 T}}{\varepsilon |\nu|} \| V \|_{X_{\varepsilon \hat{\nu}}(\Sigma)} \| u^\sharp\|_{L^\infty(\Sigma)}.
    \end{align*}
    In the last inequality we have used that $|\nu| \geq 2 c_V$.
    In order to conclude the proof this proposition is enough to realize that \eqref{eq:uflat} is equivalent to \eqref{eq:Neumann}.
\end{proof}
%
%
%
%
%
%
\begin{corollary}\label{cor:regularity_uflat}\sl 
Consider $V \in L^\infty (\Sigma) $ such that
\begin{equation*}
\int_{(0,T)} \| e^{\delta |x|} V(t, \centerdot) \|_{L^\infty (\R^n)}^2 \, \dd t < \infty
\end{equation*}
for some $\delta > 0$. Then, the solution $u^\flat \in X_{\varepsilon \hat \nu} (\Sigma)$, with
$\varepsilon \in (0, \delta)$, announced in \Cref{Prop CGO1}
can be represented by a function in $C([0, T]; L^2(\R^n))\cap L^2((0,T); \dot H^1(\R^n))$ such that 
$u^\flat(0,\centerdot) = 0$.\footnote{For a function $f \in \mathcal{S}(\R^n)$ we have that
$\| f \|_{\dot H^1 (\R^n)} = \| \nabla f  \|_{L^2(\R^n)}$.}
\end{corollary}

\begin{proof}
By \eqref{eq:uflat} and the fact that $S_\nu$ is a 
bounded operator from $L^1((0, T); L^2(\R^n))$ to 
$C([0, T]; L^2(\R^n))$, we know that, if $V(u^\sharp + u^\flat)$ belongs to 
$L^1((0, T); L^2(\R^n))$, we could choose
$u^\flat$ to be represented by a function in $C([0, T]; L^2(\R^n))$ so that
$u^\flat(0,\centerdot) = 0$.
Let us see that this is in fact the case:
\begin{align*}
\| V(u^\sharp + u^\flat) \|_{L^1((0,T); L^2(\R^n))} \leq & \| e^{\delta |\hat \nu \cdot x|} V \|_{L^1((0,T); L^\infty(\R^n))} \| e^{-\delta |\hat \nu \cdot x|} u^\sharp \|_{C([0, T]; L^2(\R^n))}\\
& \quad + \| e^{- \varepsilon \hat{\nu} \cdot \x} V \|_{L^2((0,T); L^\infty(\R^n))} \| u^\flat \|_{X_{\varepsilon \hat \nu} (\Sigma)}.
\end{align*}
Recall \eqref{in:usharp_full-space} and note that
\begin{align*}
& \| e^{\delta |\hat \nu \cdot x|} V \|_{L^1((0,T); L^\infty(\R^n))} \leq T^{1/2} \| e^{\delta |x|} V \|_{L^2((0,T); L^\infty(\R^n))}, \\
& \| e^{- \varepsilon \hat{\nu} \cdot \x} V \|_{L^2((0,T); L^\infty(\R^n))} \leq \| e^{\delta |x|} V \|_{L^2((0,T); L^\infty(\R^n))}.
\end{align*}

We now prove the estimate
\[
    \|S_\nu\phi\|_{L^2((0,T);\dot H^1(\mathbb R^n))}
    \lesssim
    \|\phi\|_{L^1((0,T);L^2(\mathbb R^n))}.
\]
In order to do that, we write
\[
    S_\nu\phi(t, \centerdot)
    =
    \int_0^t P_{t-s}^\nu[\phi(s, \centerdot)]\, \dd s.
\]
with
\[
    P_r^\nu h(x)
    =
    \int_{\mathbb R^n}
    H_n(r,x-y+2r\nu)h(y)\,\dd y
\]
for $r>0$.

Since
\[
    P_r^\nu h(x)
    =
    (e^{r\Delta}h)(x+2r\nu),
\]
and translations are isometries on $L^2(\mathbb R^n)$, we have
\[
    \|\nabla P_r^\nu h\|_{L^2(\mathbb R^n)}
    =
    \|\nabla e^{r\Delta}h\|_{L^2(\mathbb R^n)}.
\]

By Minkowski's integral inequality,
\begin{align*}
    \|S_\nu\phi\|_{L^2((0,T);\dot H^1(\mathbb R^n))}
    &=
    \left(
        \int_0^T
        \left\|
            \int_0^t
            \nabla P_{t-s}^\nu[\phi(s, \centerdot)]\, \dd s
        \right\|_{L^2(\mathbb R^n)}^2
        \dd t
    \right)^{1/2}
    \\
    &\le
    \int_0^T
    \left(
        \int_s^T
        \|\nabla P_{t-s}^\nu[\phi(s, \centerdot)]\|_{L^2(\mathbb R^n)}^2
        \dd t
    \right)^{1/2}
    \dd s.
\end{align*}
For fixed $s\in(0,T)$, setting $r=t-s$ and using the
translation invariance above gives us
\begin{align*}
    \int_s^T
    \|\nabla P_{t-s}^\nu[\phi(s, \centerdot)]\|_{L^2(\mathbb R^n)}^2
    \dd t
    &=
    \int_0^{T-s}
    \|\nabla e^{r\Delta}[\phi(s, \centerdot)]\|_{L^2(\mathbb R^n)}^2
    \dd r
    \\
    &\le
    \int_0^\infty
    \|\nabla e^{r\Delta}[\phi(s, \centerdot)]\|_{L^2(\mathbb R^n)}^2
    \dd r.
\end{align*}

We next use Plancherel's theorem. Writing 
$\widehat \phi(s, \xi) = \mathcal{F}[\phi(s, \centerdot)](\xi)$, we obtain
\begin{align*}
    \int_0^\infty
    \|\nabla e^{r\Delta}[\phi(s, \centerdot)]\|_{L^2(\mathbb R^n)}^2
    \dd r
    &=
    \int_0^\infty
    \int_{\mathbb R^n}
    |\xi|^2 e^{-2r|\xi|^2}
    |\widehat \phi(s, \xi)|^2
    \,\dd \xi\,\dd r
    \\
    &=
    \int_{\mathbb R^n}
    |\xi|^2 |\widehat \phi(s,\xi)|^2
    \left(
        \int_0^\infty e^{-2r|\xi|^2}\,\dd r
    \right)
    \dd \xi
    \\
    &=
    \frac12
    \int_{\mathbb R^n}
    |\widehat \phi(s,\xi)|^2\,\dd \xi
    \\
    &=
    \frac12
    \|\phi(s, \centerdot)\|_{L^2(\mathbb R^n)}^2.
\end{align*}

Consequently,
\[
    \left(
        \int_s^T
        \|\nabla P_{t-s}^\nu[\phi(s, \centerdot)]\|_{L^2(\mathbb R^n)}^2
        \dd t
    \right)^{1/2}
    \le
    \frac{1}{\sqrt{2}}
    \|\phi(s, \centerdot)\|_{L^2(\mathbb R^n)}.
\]

Substituting this estimate into the Minkowski inequality yields
\begin{align*}
    \|S_\nu\phi\|_{L^2((0,T);\dot H^1(\mathbb R^n))}
    &\le
    \frac{1}{\sqrt{2}}
    \int_0^T
    \|\phi(s, \centerdot)\|_{L^2(\mathbb R^n)}
    \,\dd s
    \\
    &=
    \frac{1}{\sqrt{2}}
    \|\phi\|_{L^1((0,T);L^2(\mathbb R^n))}.
\end{align*}

Therefore,
\[\|S_\nu \phi \|_{L^2((0, T); \dot H^1(\R^n))} \lesssim \| \phi\|_{L^1((0,T); L^2(\R^n))}.\]
This together with \eqref{in:Snu_trivial} ensure that $S_\nu$ is 
bounded from $L^1((0, T); L^2(\R^n))$ to 
$C([0, T]; L^2(\R^n))\cap L^2((0,T); \dot H^1(\R^n))$.
As a consequence, $u^\flat$ can be chosen to be represented
by a function in $C([0, T]; L^2(\R^n))\cap L^2((0,T); \dot H^1(\R^n))$.
\end{proof}

\section{Proof of \texorpdfstring{\Cref{th:uniqueness}}{th:uniqueness}}
\label{sec:proof_main-th}
Consider $V_1, V_2 \in L^\infty (\Sigma) $ so that,
there is an $\varepsilon > 0$ for which
\[\int_\Sigma e^{ \varepsilon |x|} |V_j(t, x) |^2 \, \dd (t, x) < \infty\]
with $j \in \{ 1, 2 \}$. Introduce $W_2 (t, x) = V_2 (T-t, x)$ for all
$(t, x) \in \Sigma$.

As discussed in \Cref{sec:CGO}, 
we can construct
\begin{align*}
& u_1 = e^{|\nu|^2 \tm + \nu \cdot \x} (u_1^\sharp + u_1^\flat)\\
& u_2 = e^{|\nu|^2 \tm - \nu \cdot \x} (u_2^\sharp + u_2^\flat),
\end{align*}
solutions for the equations
\begin{equation*}
(\partial_\tm - \Delta - V_1) u_1 = (\partial_\tm - \Delta - W_2) u_2 = 0 \enspace \textnormal{in} \enspace \Sigma,
\end{equation*}
with $\nu \in \R^n $ such that 
$|\nu| \geq e^{\varepsilon^2 T} \max\big(\| V_1 \|_{L^\infty (\Sigma)}, \| V_2 \|_{L^\infty (\Sigma)}\big)/\varepsilon$. Here
\begin{equation*}
u_j^\sharp (t, x) = \frac{1}{(2\pi)^\frac{n-1}{2}} \int_{H_{\hat{\nu}}} e^{ix \cdot \xi} e^{- t |\xi|^2} {\psi_j}(\xi) \, \dd \sigma_{\hat{\nu}} (\xi) \quad \forall (t, x) \in (0,\infty) \times \R^n.
\end{equation*}
with $\psi_j \in \mathcal{S}(\R^n)$.
For coherence with the choice $-\nu$ for $u_2$, we should have 
written, in the definition of $u_2^\sharp$, $H_{-\hat \nu}$ and $\sigma_{-\hat \nu}$
instead of $H_{\hat \nu}$ and $\sigma_{\hat \nu}$. But observe that 
$H_{\hat \nu} = H_{-\hat \nu}$ and $\sigma_{\hat \nu} = \sigma_{-\hat \nu}$.
Recall that
\[\sup_{t \in [0, T]} \| u_j^\sharp (t, \centerdot) \|_{L^\infty(\R^n)} \lesssim 1.\]
The implicit constant here depends on $\| \psi_j \|_{L^1(H_{\hat \nu})}$.
Additionally, $u_j^\flat$ is a measurable function such that
\[ \| e^{\varepsilon \hat{\nu} \cdot \x} u_1^\flat \|_{L^2(\Sigma)} + \| e^{- \varepsilon \hat{\nu} \cdot \x} u_2^\flat \|_{L^2(\Sigma)} \lesssim \frac{1}{|\nu|}. \]
Here the implicit constant depends on 
$\varepsilon$, $T$ and the quantities $\int_\Sigma e^{ \varepsilon |\x|} |V_j|^2$ and
$\| \psi_j \|_{L^1(H_{\hat \nu})}$.

Now, we introduce $v_2(t, x) = u_2(T-t,x)$ for every $(t, x) \in \Sigma$
which is solution for 
\begin{equation*}
(\partial_\tm + \Delta + V_2) v_2 = 0 \enspace \textnormal{in} \enspace \Sigma.
\end{equation*}
Note that
\[v_2 = e^{|\nu|^2 (T - \tm) - \nu \cdot \x} (v_2^\sharp + v_2^\flat)\]
with $v_2^\sharp (t, x) = u_2^\sharp (T-t,x)$, $v_2^\flat (t, x) = u_2^\flat (T-t,x)$,
\[ \sup_{t \in [0, T]} \| v_2^\sharp (t, \centerdot) \|_{L^\infty(\R^n)} \lesssim 1, \textnormal{ and } \| e^{- \varepsilon \hat{\nu} \cdot \x} v_2^\flat \|_{L^2(\Sigma)} \lesssim \frac{1}{|\nu|}; \]
with the corresponding the implicit constant depending on the same quantities already 
discussed.

We would like to plug these solutions into the orthogonality relation derived from
$\mathcal{U}_T^1 = \mathcal{U}_T^2$ and \Cref{prop:integral_formula}. Unfortunately, they grow
exponentially and do not belong to $C([0,T]; L^2(\R^n))$. Then, our next step is to extend the
orthogonality relation for exponentially growing solutions. In order to undertake this task,
it is convenient to realize that $u_1^\flat$ and $v_2^\flat$ can be represented by functions in
$C([0, T]; L^2(\R^n))\cap L^2((0,T); \dot H^1(\R^n))$. This is justified in \Cref{cor:regularity_uflat} under the 
assumption 
\begin{equation*}
\int_{(0,T)} \| e^{\delta |x|} V(t, \centerdot) \|_{L^\infty (\R^n)}^2 \, \dd t < \infty
\end{equation*}
with $\delta > \varepsilon$.

\begin{proposition}\label{prop:extension-ortho}\sl 
Consider $V_1$ and $V_2 $ in $ L^\infty (\Sigma)$
such that
\[\int_{(0,T)} \| e^{c |x|} V(t, \centerdot) \|_{L^\infty (\R^n)}^2 \, \dd t < \infty\]
for some $c>0$.
Then, we can construct 
\begin{align*}
& u_1^\nu = e^{|\nu|^2 \tm + \nu \cdot \x} (u_1^\sharp + u_1^\flat)\\
& v_2^\nu = e^{|\nu|^2 (T - \tm) - \nu \cdot \x} (v_2^\sharp + v_2^\flat)
\end{align*}
for any $\nu \in \R^n $ such that 
$|\nu| \geq e^{\varepsilon^2 T} \max\big(\| V_1 \|_{L^\infty (\Sigma)}, \| V_2 \|_{L^\infty (\Sigma)}\big)/\varepsilon$ with $\varepsilon \in (0, c)$.

If $\mathcal{U}_{T}^{1}=\mathcal{U}_{T}^{2}$ and 
$c > e^{1/2} \max\big(\| V_1 \|_{L^\infty (\Sigma)}, \| V_2 \|_{L^\infty (\Sigma)}\big)/(2T)^{1/2}$, then
\begin{align}\label{17}
   \int_\Sigma (V_1 - V_2) u_1^\nu v_2^\nu \, =0
\end{align}
whenever $|\nu| < c$.
\end{proposition}

\begin{proof}
We start by proving $\mathcal{U}_{T}^{1}=\mathcal{U}_{T}^{2}$ implies that
\begin{equation}\label{EXP-PHYSICAL}
\int_{\Sigma}\left(V_{1}-V_{2}\right) u_{1} v_{2}^{\nu} \, =0,
\end{equation}
for every $u_{1}\in C([0,T]; L^2(\R^n))$ solution of the $(\partial_\tm - \Delta - V_{1})u_1 = 0$ in $\Sigma$, and every exponentially-growing solution $v_{2}^{\nu}$ with $|\nu| < c$.

In order to do so, we first notice that $(V_1 - V_2) u_1 v_2^\nu \in L^1(\Sigma)$
whenever $|\nu| < c$, since $ e^{\nu \cdot \x + \varepsilon |\hat \nu \cdot \x|} (V_1 - V_2)$
belongs to $L^1((0,T); L^\infty(\R^n))$, while $u_1$, $v_2^\flat$ and $e^{-\varepsilon |\hat \nu \cdot \x| v_2^\sharp}$ belong to $C([0, T] ; L^{2}(\mathbb{R}^{n}))$.

Given $u_{1}$, let $w_{2} \in C([0, T] ; L^{2}(\mathbb{R}^{n}))$ be the solution of 
$$
\begin{cases}\left( \partial_{\tm}-\Delta-V_{2}\right) w_{2}=\left(V_{1}-V_{2}\right) u_{1} & \text { in } \Sigma, \\ w_{2}(0, \centerdot)=0 & \text { in } \mathbb{R}^{n}.\end{cases}
$$
Note that $(V_{1}-V_{2})u_{1} \in L^{1}((0, T) ; L^{2}(\mathbb{R}^{n}))$. Moreover, by 
\Cref{prop:integral_formula} and the fact that $\mathcal{U}_{T}^{1}=\mathcal{U}_{T}^{2}$,
we have that $(V_{1}-V_{2})u_{1}$ satisfies the same conditions as $F$ in 
\Cref{Zero Trace 1}. Consequently, $w_2(T, \centerdot)= 0$.

We have that
\[\int_{\Sigma}\left(V_{1}-V_{2}\right) u_{1} v_{2}^{\nu} \, = \int_{\Sigma}\left( \partial_{\tm}-\Delta-V_{2}\right) w_{2} v_{2}^{\nu} \, = \int_{\Sigma}\left( \partial_{\tm}-\Delta\right) w_{2} v_{2}^{\nu} \, - \int_{\Sigma} V_{2} w_{2} v_{2}^{\nu} \, .
\]
Last identity holds because $\left( \partial_{\tm}-\Delta\right) w_{2} v_{2}^{\nu}$ and 
$V_{2} w_{2} v_{2}^{\nu}$ belongs to $L^1(\Sigma)$ whenever $|\nu| < c$. We know that 
$V_{2} w_{2} v_{2}^{\nu} \in L^1(\Sigma)$ because 
$ e^{\nu \cdot \x + \varepsilon |\hat \nu \cdot \x|} V_2$ belongs to $L^1((0,T); L^\infty(\R^n))$, while $w_2$, $v_2^\flat$ and 
$e^{-\varepsilon |\hat \nu \cdot \x| v_2^\sharp}$ belong to
$C([0, T] ; L^{2}(\mathbb{R}^{n}))$. Consequently, 
$\left( \partial_{\tm}-\Delta\right) w_{2} v_{2}^{\nu} \in L^1(\Sigma)$ since
$(\partial_{\tm}-\Delta)w_2 = V_{2} w_{2} + \left(V_{1}-V_{2}\right) u_{1}$.

By \Cref{cor:vsharp-vflat}
\begin{align*}
\int_{\Sigma}\left( \partial_{\tm}-\Delta\right)& w_{2} v_{2}^{\nu} \, \\
& = \int_{\Sigma}\left( \partial_{\tm}-\Delta\right) w_{2} \, e^{|\nu|^2 (T - \tm) - \nu \cdot \x} v_2^\sharp \, + \int_{\Sigma}\left( \partial_{\tm}-\Delta\right) w_{2} \, e^{|\nu|^2 (T - \tm) - \nu \cdot \x} v_2^\flat \, \\
& = - \int_{\Sigma} w_2 \left( \partial_{\tm}+\Delta\right)(e^{|\nu|^2 (T - \tm) - \nu \cdot \x} v_2^\flat) = \int_{\Sigma} V_{2} w_{2} v_{2}^{\nu} \, .
\end{align*}
Last identity holds because 
$\left( \partial_{\tm}+\Delta\right)(e^{|\nu|^2 (T - \tm) - \nu \cdot \x} v_2^\sharp) = 0$ and
$-\left( \partial_{\tm}+\Delta\right) v_2^\nu = V_2 v_2^\nu$. Consequently, \eqref{EXP-PHYSICAL} holds.

Finally, we prove that \eqref{EXP-PHYSICAL} implies that
\begin{equation}\label{EXP-EXP}
\int_{\Sigma}\left(V_{1}-V_{2}\right) u_{1}^{\nu} v_{2}^{\nu}=0,
\end{equation}
holds for every exponentially-growing solutions $u_1^\nu$ and $v_{2}^{\nu}$ with $|\nu| < c$.

We obviously have that $(V_1 - V_2) u^\nu_1 v_2^\nu \in L^1(\Sigma)$ since
$u^\nu_1 v_2^\nu = e^{|\nu|^2 T} (u^\flat_1 + u^\sharp) (v_2^\sharp + v_2^\flat)$ and
$ e^{2 \varepsilon |\hat \nu \cdot \x|} (V_1 - V_2)$
belongs to $L^1((0,T); L^\infty(\R^n))$, while $u_1^\flat$, $v_2^\flat$, 
$e^{-\varepsilon |\hat \nu \cdot \x| u_1^\sharp}$ and 
$e^{-\varepsilon |\hat \nu \cdot \x| v_2^\sharp}$ belong to
$C([0, T] ; L^{2}(\mathbb{R}^{n}))$.

Now, let $w_{1} \in C([0, T] ; L^{2}(\mathbb{R}^{n}))$ be the solution to 
$$
\begin{cases}\left( \partial_{\tm}+\Delta+V_{1}\right) w_{1}=\left(V_{1}-V_{2}\right)v_{2}^{\nu} & \text { in } \Sigma \\ w_{1}(T, \centerdot)=0 & \text { in } \mathbb{R}^{n}.\end{cases}
$$
Note that $(V_{1}-V_{2})v_2^\nu \in L^{1}((0, T) ; L^{2}(\mathbb{R}^{n}))$.  Moreover, by 
the identity \eqref{EXP-PHYSICAL},
we have that $(V_{1}-V_{2})v_2^\nu$ satisfies the same conditions as $G$ in 
\Cref{Zero Trace 2}. Consequently, $w_1(0, \centerdot)= 0$.

We have that
\[\int_{\Sigma}\left(V_{1}-V_{2}\right) u_{1}^\nu v_{2}^{\nu} \, = \int_{\Sigma}\left( \partial_{\tm}+\Delta+V_1\right) w_1 u_1^{\nu} \, = \int_{\Sigma}\left( \partial_{\tm}+\Delta\right) w_1 u_1^{\nu} \, + \int_{\Sigma} V_1 w_1 u_1^{\nu} \, .
\]
Last identity holds because $\left( \partial_{\tm}+\Delta\right) w_1 u_1^{\nu}$ and 
$V_1 w_1 u_1^{\nu}$ belongs to $L^1(\Sigma)$.

By \Cref{lem:usharp,lem:uflat}, we have
\begin{align*}
\int_{\Sigma}\left( \partial_{\tm}+\Delta\right) w_1 u_1^{\nu} &=
\int_{\Sigma}\left( \partial_{\tm}+\Delta\right) w_1 e^{|\nu|^2 \tm + \nu \cdot \x} u_1^\sharp \, + \int_{\Sigma}\left( \partial_{\tm}+\Delta\right) w_1 e^{|\nu|^2 \tm + \nu \cdot \x} u_1^\flat \, \\
&= - \int_{\Sigma} w_1 \left( \partial_{\tm}-\Delta\right) (e^{|\nu|^2 \tm + \nu \cdot \x} u_1^\flat) \, = - \int_{\Sigma} V_1 w_1 u_1^{\nu} \, .
\end{align*}
Last identity holds because 
$\left( \partial_{\tm} - \Delta\right)(e^{|\nu|^2 \tm + \nu \cdot \x} u_1^\sharp) = 0$ and
$\left( \partial_{\tm} - \Delta\right) u_1^\nu = V_1 u_1^\nu$. Consequently, the identity \eqref{EXP-EXP} holds.
\end{proof}

Therefore, the solutions $u_1$ and $v_2$ can be plugged into \eqref{17}, 
and multiplying by $e^{-|\nu|^2 T}$
 we get 
%
\begin{align}\label{ineq 18}
   \Big| \int_\Sigma   (V_1 - V_2) u_1^\sharp {v_2^\sharp}  \Big|  \lesssim \frac{1}{|\nu|}.
\end{align}
 Here the implicit constant depends on 
$\varepsilon$ and $T$ as well as on $\| \psi_j \|_{L^1(H_{\hat \nu})}$,
$\int_\Sigma e^{ \varepsilon |\x|} |V_j|^2$ and $ \| V_j  \|_{L^\infty(\Sigma)}  $, for $j=1,2$. Indeed, by assumption we made on $V_j$ we get the following three estimates 
\begin{align*}
   \Big| \int_\Sigma   (V_1 - V_2) u_1^\sharp {v_2^\flat} \Big| & =  \Big| \int_\Sigma   e^{\varepsilon \hat{\nu} \cdot \x}(V_1 - V_2) u_1^\sharp \big(e^{- \varepsilon \hat{\nu} \cdot \x}v_2^\flat \big) \Big|  \\
   \lesssim & \, \| e^{\varepsilon | \x|}(V_1 - V_2)  \|_{L^2(\Sigma)} 
   \| u_1^\sharp  \|_{L^\infty(\Sigma)} 
   \| e^{- \varepsilon \hat{\nu} \cdot \x}v_2^\flat   \|_{L^2(\Sigma)} \lesssim \frac{1}{|\nu|}, \\
      \Big| \int_\Sigma   (V_1 - V_2) u_1^\flat {v_2^\sharp}  \Big| &=   \Big| \int_\Sigma   e^{-\varepsilon \hat{\nu} \cdot \x}(V_1 - V_2) \big(e^{\varepsilon \hat{\nu} \cdot \x} u_1^\flat \big) v_2^\sharp  \Big|  \\
   &\lesssim \| e^{\varepsilon | \x|}(V_1 - V_2)  \|_{L^2(\Sigma)}  
   \| e^{\varepsilon \hat{\nu} \cdot \x}u_1^\flat   \|_{L^2(\Sigma)}
    \| v_2^\sharp  \|_{L^\infty(\Sigma)} \lesssim \frac{1}{|\nu|},\\
         \Big| \int_\Sigma   (V_1 - V_2) u_1^\flat {v_2^\flat}  \Big| &=   \Big| \int_\Sigma (V_1 - V_2) \big(e^{\varepsilon \hat{\nu} \cdot \x} u_1^\flat \big)\big(e^{- \varepsilon \hat{\nu} \cdot \x}v_2^\flat \big) \Big|  \\
   \lesssim  \| V_1 &- V_2  \|_{L^\infty(\Sigma)}  
   \| e^{\varepsilon \hat{\nu} \cdot \x}u_1^\flat   \|_{L^2(\Sigma)}
    \| e^{- \varepsilon \hat{\nu} \cdot \x}v_2^\flat \|_{L^2(\Sigma)}\lesssim \frac{1}{|\nu|^2}.
\end{align*}
Combining these, we obtain \eqref{ineq 18}. Therefore by letting $|\nu|$ goes to the infinity we end up by getting
\begin{align}
    \int_\Sigma  (V_1 - V_2) u_1^\sharp {v_2^\sharp}=0.
\end{align}
Continue by recalling that
\begin{align*}
u_1^\sharp (t, x) & = \frac{1}{(2\pi)^\frac{n-1}{2}} \int_{H_{\hat{\nu}}} e^{ix \cdot \xi} e^{- t |\xi|^2} {\psi_1}(\xi) \, \dd \sigma_{\hat{\nu}} (\xi) & & \forall(t, x) \in \Sigma, \\
v_2^\sharp (t, x) & = \frac{1}{(2\pi)^\frac{n-1}{2}} \int_{H_{\hat{\nu}}} e^{i x \cdot \eta} e^{-(T-t)|\eta|^2} {\psi_2}(\eta) \, \dd \sigma_{\hat{\nu}} (\eta) & & \forall(t, x) \in \Sigma.
\end{align*}
Writing
\begin{equation*}
	F (t, x) = \left\{
		\begin{aligned}
		& V_1 (t, x) - V_2 (t, x) &  & \textnormal{if} \enspace (t, x) \in \Sigma \\
		& 0 &  & \textnormal{if} \enspace (t, x) \notin \Sigma,
		\end{aligned}
	\right.
\end{equation*}
and using the fact that 
$\psi_1$ and $\psi_2$ can be chosen arbitrarily in $\mathcal{S}(\R^n)$, we have
$$
\int_{\mathbb{R} \times \mathbb{R}^n} F(t, x) e^{t(|\eta|^2-|\xi|^2)}   {e^{i x \cdot (\xi+\eta)} } \, \dd (t,x) = 0,
$$
for all $\xi, \eta\in \R^n$ such that $\eta \cdot \nu = \xi \cdot \nu=0$.
For every $(\tau, \kappa) \in \mathbb{R} \times \mathbb{R}^n$ such that $\kappa \neq 0$, we 
choose $\nu \in \mathbb{R}^n \backslash\{0\}$ so that $\kappa \cdot \nu=0$. Then, we choose
\begin{align*}
   \xi=- \frac{1}{2}(1-\frac{\tau}{|\kappa|^2})\kappa   ~~~\text{and}~~~ \eta= -\frac{1}{2}(1+\frac{\tau}{|\kappa|^2})\kappa.
\end{align*}
Which imply $|\eta|^2-|\xi|^2=\tau$ and $\xi+\eta=- \kappa$. Then, for every $\kappa \in \R^n \setminus \{ 0 \}$ we have that
\begin{align*}
\int_\R e^{t\tau} G_\kappa (t) \,\dd t = 0 \quad \forall \tau \in \R,
\end{align*}
with $ G_\kappa (t) = \mathcal{F}[F(t, \centerdot)](\kappa) $---the Fourier transform of $F$
with respect to the spatial variable.
Proceeding now exactly as the last part of density argument in
\Cref{sec: intro} we get $F(t, x)=0$ for a.e. $(t, x) \in \mathbb{R} \times \mathbb{R}^n$.
This completes the proof.

\section{Boundedness of \texorpdfstring{$S_\nu$}{Snu} in non translation-invariant spaces}
\label{sec:boundedness}
We define the $1$-dimensional heat kernel as
\begin{equation}
\label{id:1d_heat_kernel}
H: (t, x) \in (0,\infty) \times \R \longmapsto \frac{1}{(4\pi t)^{1/2}} e^{-\frac{x^2}{4t}}.
\end{equation}

In the following lemma, we revisit the two-sided Laplace transform of 
$x \in \R \mapsto H(t, x + \rho t)$, or equivalently, 
the moment generating function of a Gaussian random variable
with mean $-\rho t$ and variance $2t$.
\begin{lemma}\label{lem:heat-exponential}\sl
For $\lambda \in \R$, we have that
\[ \| e^{\lambda \x} H(t, \centerdot + \rho t) \|_{L^1(\R)} = e^{-\lambda\rho t + \lambda^2 t} \]
for all $t \in (0, \infty)$ and $\rho \in \R$.
\end{lemma}

\begin{proof}
Start by observing that
\[ \| e^{\lambda \x} H(t, \centerdot + \rho t) \|_{L^1(\R)} = e^{- \lambda \rho t} \| e^{\lambda \x} H(t, \centerdot) \|_{L^1(\R)}. \]
Furthermore, since
\[-\frac{x^2}{4t} + \lambda x = - \frac{1}{4t} \Big( x^2 - 4t \lambda x \Big) = - \frac{(x-2t\lambda)^2}{4t} + t \lambda^2, \]
we have that
\[ e^{\lambda x} H(t, x) = e^{\lambda^2t} H(t, x - 2t\lambda), \]
and consequently,
\[ \| e^{\lambda \x} H(t, \centerdot) \|_{L^1(\R)} = e^{\lambda^2t} \| H(t, \centerdot - 2t\lambda) \|_{L^1(\R)} = e^{\lambda^2t}. \]
With this last computation, we obtain the identity of the statement.
\end{proof}

\begin{remark} The operator $S_\nu$ is originally defined so that its inputs and outputs take values $\overline{\Sigma}$. However, in the next theorem we let $S_\nu$ also have
inputs and outputs in $(0, \infty) \times \R^n$.
\end{remark}

\begin{theorem}\label{thm:key-inequality}\sl Consider $\lambda > 0$ and $\nu \in \R^n \setminus\{ 0 \}$, and
write $\hat{\nu} = \nu/|\nu|$.
Then, 
\[\| e^{-\lambda^2 \tm + \lambda \hat{\nu} \cdot \x} S_\nu \phi \|_{L^2((0, \infty) \times \R^n)} \leq \frac{1}{2 \lambda |\nu|} \| e^{-\lambda^2 \tm + \lambda \hat{\nu} \cdot \x} \phi \|_{L^2((0, \infty) \times \R^n)} \]
for all $\phi \in \mathcal{D}((0, \infty) \times \R^n)$.
\end{theorem}
\begin{proof}
Consider $Q \in \mathrm{O}(n)$ such that $ \nu=|\nu| Q e_n $, then
$$
\| e^{\lambda \hat{\nu} \cdot \x} S_\nu \phi (t, \centerdot) \|_{L^2(\R^n)}=\| e^{\lambda  \x_n} S_\nu \phi (t, Q \centerdot) \|_{L^2(\R^n)} \quad \forall t \in (0, \infty).
$$
Note that
\[ S_\nu \phi(t, Qx) = \int_0^t \Big( \int_{\R^n} H_n (t-s, Q (x - y) + 2(t-s) \nu) \phi(s, Qy) \, \dd y \Big) \, \dd s, \]
after the change introducing $Qy$ in the integrand. Then, writing
$\psi(s,y)=\phi(s, Qy)$, for all $(t,x) \in (0, \infty) \times \R^n$, and using the fact that
the heat kernel $H_n$ is radially symmetric:
$$
H_n (t-s, Q (x - y) + 2(t-s) \nu) = H_n (t-s, x - y + 2(t-s) |\nu| e_n);
$$
we obtain that
$$
S_\nu \phi(t, Qx)= S_{|\nu| e_n} \psi(t, x) \quad \forall(t, x) \in (0, \infty) \times \mathbb{R}^n
$$
and consequently,
$$
\| e^{\lambda \hat{\nu} \cdot \x} S_\nu \phi (t, \centerdot) \|_{L^2(\R^n)}=\| e^{\lambda  \x_n} S_{|\nu| e_n} \psi (t, \centerdot) \|_{L^2(\R^n)} \quad \forall t \in (0, \infty).
$$
Since
\begin{align*}
e^{\lambda x_n} S_{|\nu| e_n} & \psi(t, x)\\
& = \int_0^t \Big( \int_{\R^n} e^{\lambda (x_n-y_n)} H_n (t-s, x - y + 2(t-s) |\nu| e_n) e^{\lambda y_n} \psi(s,y) \, \dd y \Big) \, \dd s,
\end{align*}
by Young's inequality we derive that
\begin{align*}
\| e^{\lambda  \x_n} S_{|\nu| e_n} \psi (t, \centerdot) & \|_{L^2(\R^n)} \\
\leq & \int_0^t \| e^{\lambda \x_n} H_n (t-s, \centerdot + 2(t-s) |\nu| e_n)\|_{L^1(\R^n)} \| e^{\lambda \x_n} \psi(s, \centerdot)\|_{L^2(\R^n)} \, \dd s.
\end{align*}
Observe that the heat kernel can be split as follows:
\[H_n (t-s, x + 2(t-s) |\nu| e_n) = H_{n-1} (t-s, x^\prime) H (t-s, x_n + 2(t-s) |\nu|)\]
with $x = (x^\prime, x_n)$.
Then, we get
\begin{align*}
\| & e^{\lambda \x_n} H_n (t-s, \centerdot + 2(t-s) |\nu| e_n)\|_{L^1(\R^n)} \\
& = \|H_{n-1} (t-s, \centerdot)\|_{L^1(\R^{n-1})} \|e^{\lambda \x_n} H (t-s, \centerdot + 2(t-s) |\nu|)\|_{L^1(\R)} = e^{(\lambda^2-2\lambda|\nu|)(t-s)}
\end{align*}
by \eqref{id:heat_L1_tranlation-invariant} for the $(n-1)$-dimensional kernel and by
\Cref{lem:heat-exponential} with $\lambda > 0$ and $\rho = 2 |\nu|$.
%
%
%
%
%
%
Hence,
\[ \| e^{\lambda  \x_n} S_{|\nu| e_n} \psi (t, \centerdot) \|_{L^2(\R^n)} \leq \int_0^t e^{-2\lambda|\nu|(t-s)} e^{\lambda^2(t-s)}  \| e^{\lambda \x_n} \psi(s, \centerdot)\|_{L^2(\R^n)} \, \dd s,  \]
and consequently
\begin{align*}
   \| e^{-\lambda^2 t + \lambda \hat{\nu} \cdot \x} & S_\nu \phi (t, \centerdot) \|_{L^2(\R^n)} = \|e^{-\lambda^2 t + \lambda \x_n} S_{|\nu| e_n} \psi(t, \centerdot)\|_{L^2(\R^{n})}\cr
        &\leq \int_\R \mathds{1}_{(0, \infty)}(t-s) e^{-2\lambda |\nu| (t-s)} \mathds{1}_{(0, \infty)}(s)\|e^{-\lambda^2 s + \lambda \x_n}\psi(s, \centerdot)\|_{L^2(\R^{n})}   \, \dd s.
\end{align*}
Here $\mathds{1}_{(0, \infty)}$ denotes the characteristic function of the interval
$(0, \infty)$.
Therefore, applying Young's inequality and noting that 
\[ \int_0^\infty e^{-2\lambda |\nu| t} \, \dd t = \frac{1}{2\lambda  |\nu|}, \]
we conclude
\begin{equation*}
    \| e^{-\lambda^2 \tm + \lambda \hat{\nu} \cdot \x} S_\nu \phi \|_{L^2((0, \infty) \times \R^n)}  \leq \frac{1}{2 \lambda |\nu|} \| e^{-\lambda^2 \tm + \lambda \x_n} \psi \|_{L^2((0, \infty) \times \R^n)}.
\end{equation*}
The desired result now follows from the fact that 
\begin{equation*}
   \| e^{-\lambda^2 \tm + \lambda \x_n} \psi \|_{L^2((0, \infty) \times \R^n)} =  \| e^{-\lambda^2 \tm + \lambda \hat{\nu} \cdot \x}  \phi \|_{L^2((0, \infty) \times \R^n)}.
\end{equation*}
This completes the proof.
\end{proof}

\begin{corollary}\sl Consider $\lambda > 0$ and $\nu \in \R^n \setminus\{ 0 \}$,
then
\[ \| e^{\lambda \hat{\nu} \cdot \x} S_\nu \phi \|_{L^2(\Sigma)} \leq \frac{e^{\lambda^2 T}}{2 \lambda |\nu|} \| e^{\lambda \hat{\nu} \cdot \x} \phi \|_{L^2(\Sigma)} \]
for all $\phi \in \mathcal{D}(\Sigma)$.
\end{corollary}

\begin{proof}
For $\phi \in \mathcal{D}(\Sigma)$, let its extension by zero be also denoted by $\phi \in\mathcal{D}((0, \infty) \times \R^n) $. Then,
\begin{align*}
\| e^{\lambda \hat{\nu} \cdot \x} S_\nu \phi \|_{L^2(\Sigma)} & \leq e^{\lambda^2 T} \| e^{-\lambda^2 \tm + \lambda \hat{\nu} \cdot \x} S_\nu \phi \|_{L^2((0, \infty) \times \R^n)}\\
& \leq \frac{e^{\lambda^2 T}}{2 \lambda |\nu|} \| e^{-\lambda^2 \tm + \lambda \hat{\nu} \cdot \x} \phi \|_{L^2((0, \infty) \times \R^n)} \leq \frac{e^{\lambda^2 T}}{2 \lambda |\nu|} \| e^{\lambda \hat{\nu} \cdot \x} \phi \|_{L^2(\Sigma)}.
\end{align*}
This concludes the proof of this corollary.
\end{proof}

\section{Complements to the extension of the orthogonality relation}\label{sec:extension_ortho}

\begin{lemma}\label{Zero Trace 1}\sl Consider $V \in L^1 ((0, T) ; L^\infty (\mathbb{R}^{n}))$ and $F \in L^{1}((0, T) ; L^{2}(\mathbb{R}^{n}))$ such that
$$
\int_{\Sigma} Fv \,=0,
$$
whenever $v \in C([0, T] ; L^{2}(\mathbb{R}^{n}))$ is the solution of the final-value problem
$$
\begin{cases} (\partial_{\tm}+\Delta+V)v=0 & \text { in } \Sigma, \\ v(T, \centerdot)=g & \text { in } \mathbb{R}^{n},\end{cases}
$$
for $g \in L^{2}(\mathbb{R}^{n})$. Then, the solution $w \in C([0, T] ; L^{2}(\mathbb{R}^{n}))$ to the initial-value problem 
$$
\begin{cases} (\partial_{\tm}-\Delta-V)w=F & \text { in } \Sigma, \\ w(0, \centerdot)=0 & \text { in } \mathbb{R}^{n},\end{cases}
$$
satisfies that
$$
w(T, \centerdot)=0.
$$
\end{lemma}
\begin{proof}
Consider $g \in L^{2}(\mathbb{R}^{n})$, and let $v$ denote the solution of the final-value 
problem in the statement. By the assumptions on $F$ and $V$, we have that
$(\partial_\tm - \Delta)w$ and $(\partial_\tm + \Delta)v$ belong to $L^1((0,T); L^2(\R^n))$.
Since $w$ and $v$ are in $C([0, T] ; L^{2}(\mathbb{R}^{n}))$ we can apply \Cref{L_A.1} to
obtain that
$$
 \int_{\mathbb{R}^{n}} w(T, \centerdot) g=\int_{\Sigma}\left[( \partial_{\tm}-\Delta)wv+w( \partial_{\tm}+\Delta) v\right].
$$
Now, adding and subtracting $V wv $ we obtain
$$
\int_{\mathbb{R}^{n}} w(T, \centerdot) g=\int_{\Sigma} F v \, + \int_{\Sigma} w ( \partial_{t}+\Delta+V) v\, =0.
$$
Since this holds for every $ g \in L^{2}(\mathbb{R}^{n})$, we have $w(T, \centerdot)=0$.
\end{proof}

We next state without proof a symmetric version of this lemma.

\begin{lemma}\label{Zero Trace 2}\sl Consider $V \in L^{1}((0, T) ; L^{\infty}(\mathbb{R}^{n}))$ and $G \in L^{1}((0, T) ; L^{2}(\mathbb{R}^{n}))$ such that
$$
\int_{\Sigma} G u\, =0,
$$
whenever $u \in C([0, T] ; L^{2}(\mathbb{R}^{n}))$ is the solution of the initial-value problem
$$
\begin{cases} (\partial_{\tm}-\Delta-V)u=0 & \text { in } \Sigma, \\ u(0, \centerdot)=f & \text { in } \mathbb{R}^{n},\end{cases}
$$
for $f \in$ $L^{2}(\mathbb{R}^{n})$. Then, the solution $w \in C([0, T] ; L^{2}(\mathbb{R}^{n}))$ of the problem
$$
\begin{cases}(\partial_{\tm}+\Delta+V) w=G & \text { in } \Sigma,  \\ w(T, \centerdot)=0 & \text { in } \mathbb{R}^{n},\end{cases}
$$
satisfies that
$$
w(0, \centerdot)=0.
$$
\end{lemma}

\begin{lemma}\label{lem:propagation_of_decay}\sl Consider $w \in C([0, T] ; L^{2}(\mathbb{R}^{n}))$ such that $w(0, \centerdot)=0$ 
and $(\partial_{\tm}-\Delta)w $ belongs to $ L^{1}((0, T) ; L^{2}(\mathbb{R}^{n}))$.
If there exists $c > 0$ such that 
$e^{c|\x|} (\partial_{\tm}-\Delta)w $ is also in $L^{1}((0, T) ; L^{2}(\mathbb{R}^{n}))$, then
\[ \|e^{\nu \cdot \x} w\|_{C([0, T] ; L^{2}(\mathbb{R}^{n}))} \leq e^{|\nu|^2 T} \|e^{\nu \cdot \x} (\partial_{\tm}-\Delta)w\|_{L^{1}((0, T) ; L^{2}(\mathbb{R}^{n}))} \]
for all $\nu \in \R^n$ such that $|\nu| \leq c$. In particular, 
$e^{c|\x|} w \in C([0, T] ; L^{2}(\mathbb{R}^{n}))$.
\end{lemma}

\begin{proof}
The condition $e^{c|\x|} w \in C([0, T] ; L^{2}(\mathbb{R}^{n}))$ is a consequence of the inequality for all $\nu \in \R^n$ such that $|\nu| \leq c$, and the assumption
$e^{c|\x|} (\partial_{\tm}-\Delta)w $ belongs to $ L^{1}((0, T) ; L^{2}(\mathbb{R}^{n}))$.

Let us check that the inequality holds.
Note that
\[ \tilde{w}: (t, x) \in \overline{\Sigma} \longmapsto \int_0^t \Big( \int_{\R^n} H_n (t-s, x-y) (\partial_{\tm}-\Delta)w(s, y) \, \dd y \Big) \, \dd s \]
belongs to $C([0, T] ; L^{2}(\mathbb{R}^{n}))$, and satisfies that $\tilde w(0, \centerdot)=0$ 
and $(\partial_{\tm}-\Delta)\tilde w = (\partial_{\tm}-\Delta)w $. By the uniqueness of the problem
$$
\begin{cases} (\partial_{\tm}-\Delta)w=F & \text { in } \Sigma, \\ w(0, \centerdot)=0 & \text { in } \mathbb{R}^{n},\end{cases}
$$
with solutions in $C([0, T] ; L^{2}(\mathbb{R}^{n}))$ and 
$F \in L^{1}((0, T) ; L^{2}(\mathbb{R}^{n}))$, we know that $w = \tilde w$.

Since 
\[e^{-|\nu|^2(t-s) - \nu \cdot (x-y)} H_n (t-s, x-y) = H_n (t-s, x-y + 2(t - s)\nu), \]
one can check by doing a straightforward computation that
\[e^{-|\nu|^2\tm - \nu \cdot \x} w = S_\nu [e^{-|\nu|^2\tm - \nu \cdot \x} (\partial_{\tm}-\Delta)w].\]
From the inequality \eqref{in:Snu_trivial} one concludes that
\[ \|e^{-\nu \cdot \x} w\|_{C([0, T] ; L^{2}(\mathbb{R}^{n}))} \leq e^{|\nu|^2 T} \|e^{-\nu \cdot \x} (\partial_{\tm}-\Delta)w\|_{L^{1}((0, T) ; L^{2}(\mathbb{R}^{n}))}. \]
This can be obviously rephrased as the inequality in the statement.
\end{proof}

\begin{corollary}\label{cor:propag_decay}\sl Consider $w \in C([0, T] ; L^{2}(\mathbb{R}^{n}))$ such that 
$w(T, \centerdot)=0$ 
and $(\partial_{\tm}+\Delta)w $ belongs to $ L^{1}((0, T) ; L^{2}(\mathbb{R}^{n}))$.
If there exists $c > 0$ such that 
$e^{c|\x|} (\partial_{\tm}+\Delta)w $ is also in $ L^{1}((0, T) ; L^{2}(\mathbb{R}^{n}))$, then
\[ \|e^{\nu \cdot \x} w\|_{C([0, T] ; L^{2}(\mathbb{R}^{n}))} \leq e^{|\nu|^2 T} \|e^{\nu \cdot \x} (\partial_{\tm}+\Delta)w\|_{L^{1}((0, T) ; L^{2}(\mathbb{R}^{n}))} \]
for all $\nu \in \R^n$ such that $|\nu| \leq c$. In particular, $e^{c|\x|} w \in C([0, T] ; L^{2}(\mathbb{R}^{n}))$.
\end{corollary}

\begin{proof}
It follows from \Cref{lem:propagation_of_decay} applied to $\tilde w (t, x) = w (T-t, x) $ for all
$(t, x) \in \overline{\Sigma}$.
\end{proof}

We finish this section with two other lemmas and a corollary of them. These lemmas deal with the
type of exponentially growing solutions we have constructed: $u = e^{|\nu|^2 \tm + \nu \cdot \x} (u^\sharp + u^\flat)$ solution of $(\partial_\tm - \Delta - V) u = 0$ in $\Sigma$ and 
$v = e^{|\nu|^2 (T - \tm) + \nu \cdot \x} (v^\sharp + v^\flat)$ solution of 
$(\partial_\tm + \Delta + V) v = 0$ in $\Sigma$.

\begin{lemma}\label{lem:usharp}\sl
Consider $w \in C([0, T] ; L^{2}(\mathbb{R}^{n}))$ such that 
$w (0, \centerdot) = w(T, \centerdot)=0$ 
and $(\partial_{\tm}+\Delta)w $ belongs to $ L^{1}((0, T) ; L^{2}(\mathbb{R}^{n}))$.
If there exists $c > 0$ such that 
$e^{c|\x|} (\partial_{\tm}+\Delta)w $ is also in $ L^{1}((0, T) ; L^{2}(\mathbb{R}^{n}))$, then
\[\int_{\Sigma}( \partial_{\tm}+\Delta) w \, e^{|\nu|^2 \tm + \nu \cdot \x} u^\sharp \, = 0 \]
for all $\nu \in \R^n$ such that $|\nu| < c$.
\end{lemma}

\begin{proof}
Note that $( \partial_{\tm}+\Delta) w \, e^{|\nu|^2 \tm + \nu \cdot \x} u^\sharp \in L^1(\Sigma)$
since $e^{\nu \cdot \x + \varepsilon |\hat \nu \cdot \x|} ( \partial_{\tm}+\Delta) w $ belongs to 
$ L^1((0,T); L^2(\R^n))$, and $ e^{-\varepsilon |\hat \nu \cdot \x|} u^\sharp $ is in 
$C([0, T] ; L^{2}(\mathbb{R}^{n}))$.

Consider $\chi \in \mathcal{D}(\R)$ so that $\chi(s) = 1$ whenever $|s| \leq 1$.
By the dominate convergence theorem first, and then \Cref{L_A.1}, we have that
\begin{align*}
\int_{\Sigma}( \partial_{\tm}+\Delta) w \, e^{|\nu|^2 \tm + \nu \cdot \x} u^\sharp \, 
&= \lim_{R\to \infty} \int_{\Sigma}( \partial_{\tm}+\Delta) w \, \chi(\hat \nu \cdot \x/R) e^{|\nu|^2 \tm + \nu \cdot \x} u^\sharp \, \\
&= - \lim_{R\to \infty} \int_{\Sigma} w (\partial_{\tm}-\Delta) [\chi(\hat \nu \cdot \x/R) e^{|\nu|^2 \tm + \nu \cdot \x} u^\sharp] \, .
\end{align*}
We could apply \Cref{L_A.1} because
$\chi(\hat \nu \cdot \x/R) e^{|\nu|^2 \tm + \nu \cdot \x} u^\sharp \in C([0, T] ; L^{2}(\mathbb{R}^{n}))$ and $(\partial_{\tm}-\Delta) [\chi(\hat \nu \cdot \x/R) e^{|\nu|^2 \tm + \nu \cdot \x} u^\sharp] \in L^{1}((0, T) ; L^{2}(\mathbb{R}^{n}))$. This latter condition requires a more careful
analysis:
\[ (\partial_{\tm}-\Delta) [\chi(\hat \nu \cdot \x/R) e^{|\nu|^2 \tm + \nu \cdot \x} u^\sharp] = - \Big[\frac{2}{R} |\nu| \chi^\prime (\hat \nu \cdot \x/R) + \frac{1}{R^2} \chi^{\prime \prime} (\hat \nu \cdot \x/R) \Big] e^{|\nu|^2 \tm + \nu \cdot \x} u^\sharp \]
where the right-hand side clearly belongs to $C([0, T] ; L^{2}(\mathbb{R}^{n}))$.

Additionally,
\begin{align*}
\Big| \int_{\Sigma} w \Big[\frac{2}{R} |\nu| \chi^\prime (\hat \nu \cdot \x/R)& + \frac{1}{R^2} \chi^{\prime \prime} (\hat \nu \cdot \x/R) \Big] e^{|\nu|^2 \tm + \nu \cdot \x} u^\sharp \, \Big|\\
& \lesssim \frac{1}{R} \| e^{\nu \cdot \x + \varepsilon |\hat \nu \cdot \x|} w \|_{C([0, T] ; L^{2}(\mathbb{R}^{n}))}  \| e^{-\varepsilon |\hat \nu \cdot \x|} u^\sharp \|_{C([0, T] ; L^{2}(\mathbb{R}^{n}))}.
\end{align*}
The implicit constant here depends on $|\nu|$, $T$ and the choice of $\chi$. Finally,
by \Cref{cor:propag_decay} we know that $e^{c|\x|} w \in C([0, T] ; L^{2}(\mathbb{R}^{n}))$.
Hence, by taking the limit as $R$ tends to $\infty$, we have that
\[ \lim_{R\to \infty} \int_{\Sigma} w (\partial_{\tm}-\Delta) [\chi(\hat \nu \cdot \x/R) e^{|\nu|^2 \tm + \nu \cdot \x} u^\sharp] \, = 0 \]
and consequently the lemma is proved.
\end{proof}

\begin{lemma}\label{lem:uflat}\sl
Consider $w \in C([0, T] ; L^{2}(\mathbb{R}^{n}))$ such that 
$w (0, \centerdot) = w(T, \centerdot)=0$ 
and $(\partial_{\tm}+\Delta)w $ belongs to $ L^{1}((0, T) ; L^{2}(\mathbb{R}^{n}))$.
If there exists $c > 0$ such that 
$e^{c|\x|} (\partial_{\tm}+\Delta)w $ is also in $ L^{1}((0, T) ; L^{2}(\mathbb{R}^{n}))$, then
\[\int_{\Sigma}(\partial_{\tm}+\Delta) w \, e^{|\nu|^2 \tm + \nu \cdot \x} u^\flat \, = - \int_{\Sigma} w ( \partial_{\tm}-\Delta) (e^{|\nu|^2 \tm + \nu \cdot \x} u^\flat) \, .\]
for all $\nu \in \R^n$ such that $|\nu| \leq c$.
\end{lemma}

\begin{proof}
Note that
$(\partial_{\tm}+\Delta) w \, e^{|\nu|^2 \tm + \nu \cdot \x} u^\flat \in L^1(\Sigma)$
since $e^{\nu \cdot \x} ( \partial_{\tm}+\Delta) w $ belongs to 
$ L^1((0,T); L^2(\R^n))$, and $ u^\flat $ is in 
$C([0, T] ; L^{2}(\mathbb{R}^{n}))$.

Consider $\chi \in \mathcal{D}(\R)$ so that $\chi(s) = 1$ whenever $|s| \leq 1$.
By the dominate convergence theorem, first and then, \Cref{L_A.1} we have that
\begin{align*}
\int_{\Sigma}( \partial_{\tm}+\Delta) w \, e^{|\nu|^2 \tm + \nu \cdot \x} u^\flat \, 
&= \lim_{R\to \infty} \int_{\Sigma}( \partial_{\tm}+\Delta) w \, \chi(\hat \nu \cdot \x/R) e^{|\nu|^2 \tm + \nu \cdot \x} u^\flat \, \\
&= - \lim_{R\to \infty} \int_{\Sigma} w (\partial_{\tm}-\Delta) [\chi(\hat \nu \cdot \x/R) e^{|\nu|^2 \tm + \nu \cdot \x} u^\flat] \, .
\end{align*}
We could apply \Cref{L_A.1} because
$\chi(\hat \nu \cdot \x/R) e^{|\nu|^2 \tm + \nu \cdot \x} u^\flat \in C([0,T]; L^2(\R^n))$
and $(\partial_{\tm}-\Delta) [\chi(\hat \nu \cdot \x/R) e^{|\nu|^2 \tm + \nu \cdot \x} u^\flat] \in L^{1}((0, T) ; L^{2}(\mathbb{R}^{n}))$.
As in \Cref{lem:usharp}, checking this letter condition requires a more careful
analysis: one can see that
%
%
%
%
\begin{align*}
(\partial_{\tm}-\Delta) [\chi(\hat \nu \cdot \x/R) e^{|\nu|^2 \tm + \nu \cdot \x} u^\flat] =
- \Big[\frac{2}{R} |\nu| \chi^\prime (\hat \nu \cdot \x/R) + \frac{1}{R^2} \chi^{\prime \prime} (\hat \nu \cdot \x/R) \Big] e^{|\nu|^2 \tm + \nu \cdot \x} u^\flat \\
- \frac{2}{R} \chi^\prime (\hat \nu \cdot \x/R) e^{|\nu|^2 \tm + \nu \cdot \x} \hat \nu \cdot \nabla u^\flat
+ \chi(\hat \nu \cdot \x/R) (\partial_{\tm}-\Delta)(e^{|\nu|^2 \tm + \nu \cdot \x} u^\flat)
\end{align*}
The first two terms in the right-hand side belongs to 
$L^2(\Sigma)$---it might be convenient to recall at this point that 
$u^\flat \in C([0, T]; L^2(\R^n))\cap L^2((0,T); \dot H^1(\R^n))$. The last term satisfies
\begin{equation}
\label{id:RHS-uflat}
\chi(\hat \nu \cdot \x/R) (\partial_{\tm}-\Delta)(e^{|\nu|^2 \tm + \nu \cdot \x} u^\flat) = \chi(\hat \nu \cdot \x/R) e^{|\nu|^2 \tm + \nu \cdot \x} V (u^\sharp + u^\flat),
\end{equation}
which is obviously in $L^1((0,T); L^2(\R^n))$.

Finally, in order to conclude the proof of this lemma we have to check
that
\begin{align*}
& \lim_{R\to \infty} \Big| \int_{\Sigma} w \Big[\frac{2}{R} |\nu| \chi^\prime (\hat \nu \cdot \x/R) + \frac{1}{R^2} \chi^{\prime \prime} (\hat \nu \cdot \x/R) \Big] e^{|\nu|^2 \tm + \nu \cdot \x} u^\flat \Big| = 0,\\
& \lim_{R\to \infty} \Big| \int_{\Sigma} w \frac{2}{R} \chi^\prime (\hat \nu \cdot \x/R) e^{|\nu|^2 \tm + \nu \cdot \x} \hat \nu \cdot \nabla u^\flat \Big| = 0,
\end{align*}
and
\[\lim_{R\to \infty} \int_{\Sigma} w \chi(\hat \nu \cdot \x/R) (\partial_{\tm}-\Delta)(e^{|\nu|^2 \tm + \nu \cdot \x} u^\flat) \, = \int_{\Sigma} w (\partial_{\tm}-\Delta)(e^{|\nu|^2 \tm + \nu \cdot \x} u^\flat) \, .  \]
The first two limits follows by the same type of arguments showed in \Cref{lem:usharp}, using 
now that $u^\flat \in C([0, T]; L^2(\R^n))\cap L^2((0,T); \dot H^1(\R^n))$.
The third limit is a consequence of the dominate convergence theorem and the identity 
\eqref{id:RHS-uflat}.
\end{proof}

\begin{corollary}\label{cor:vsharp-vflat}\sl
Consider $w \in C([0, T] ; L^{2}(\mathbb{R}^{n}))$ such that 
$w (0, \centerdot) = w(T, \centerdot)=0$ 
and $(\partial_{\tm}+\Delta)w $ belongs to $ L^{1}((0, T) ; L^{2}(\mathbb{R}^{n}))$.
If there exists $c > 0$ such that 
$e^{c|\x|} (\partial_{\tm}+\Delta)w $ is also in $ L^{1}((0, T) ; L^{2}(\mathbb{R}^{n}))$, then
\[ \int_{\Sigma}( \partial_{\tm}-\Delta) w \, e^{|\nu|^2 (T - \tm) + \nu \cdot \x} v^\sharp \, = 0 \]
whenever $|\nu| < c$, and
\[ \int_{\Sigma}( \partial_{\tm}-\Delta) w \, e^{|\nu|^2 (T - \tm) + \nu \cdot \x} v^\flat \, = - \int_{\Sigma} w ( \partial_{\tm}+\Delta)(e^{|\nu|^2 (T - \tm) + \nu \cdot \x} v^\flat) \, \]
for all $|\nu| \leq c$.
\end{corollary}

\begin{proof}
It follows from \Cref{lem:usharp,lem:uflat} applied to $\tilde w (t, x) = w (T-t, x) $ for all
$(t, x) \in \overline{\Sigma}$.
\end{proof}


\appendix 

\section{Integration by parts}\label{app:integration-parts}
\begin{lemma}\label{L_A.1}\sl
For every $u, v \in C([0, T] ; L^2(\mathbb{R}^n))$ such that $( \partial_\tm-\Delta) u$ and $(\partial_\tm+\Delta) v$ belong to $L^1((0, T) ; L^2(\mathbb{R}^n))$, we have
\begin{align*}
  \int_{\Sigma}\left[( \partial_{\mathrm{t}}-\Delta) u v + u (\partial_{\mathrm{t}}+\Delta) v \right]= \int_{\mathbb{R}^n}\left[u(T, \centerdot) v(T, \cdot)-u(0, \cdot) v(0, \cdot)\right] .  
\end{align*}
\end{lemma}

This Lemma is a straightforward adaptation of a Proposition 4.2 in \cite{zbMATH07801151}.

\begin{proof}
By applying the dominate convergence theorem we obtain
\begin{align}\label{eq_A2}
 \int_{\Sigma}\left[( \partial_{\mathrm{t}}-\Delta) u v+u {\left( \partial_{\mathrm{t}}+\Delta\right) v}\right]=\lim _{\delta \rightarrow 0} \int_{\Sigma_\delta}\left[( \partial_{\mathrm{t}}-\Delta) u v+u {\left( \partial_{\mathrm{t}}+\Delta\right) v}\right],  
\end{align}
where $\Sigma_\delta=(\delta, T-\delta) \times \mathbb{R}^n$. The idea now is to approximate $u$ and $v$ by $u_{\varepsilon}$ and $v_{\varepsilon}$ respectively, which will be smooth in $\Sigma_\delta$ and compactly supported in space. Let $\chi \in \mathcal{S}\left(\mathbb{R}^n\right)$ be a smooth cut-off such that $0 \leq \chi(x) \leq 1$ for all $x \in \mathbb{R}^n$, supp $\chi \subset\left\{x \in \mathbb{R}^n:|x| \leq 2\right\}$ and $\chi(x)=1$ whenever $|x| \leq 1$. Let $\phi \in \mathcal{S}\left(\mathbb{R}^n\right)$ and $\psi \in \mathcal{S}(\mathbb{R})$ with compact supports and such that supp $\psi \subset\{t \in \mathbb{R}:|t| \leq 1\}, \int_{\mathbb{R}^n} \phi=\int_{\mathbb{R}} \psi=1$ and $\phi(x), \psi(t) \in[0, \infty)$ for all $x \in \mathbb{R}^n$ and $t \in \mathbb{R}$. Whenever $\varepsilon>0$, define $\chi_{\varepsilon}(x)=\chi(\varepsilon^2 x)$ for $x \in \mathbb{R}^n, \phi_{\varepsilon}(x)=\varepsilon^{-n} \phi(x / \varepsilon)$ for $x \in \mathbb{R}^n$, and $\psi_{\varepsilon}(t)=\varepsilon^{-1} \psi(t / \varepsilon)$ for $t \in \mathbb{R}$.

For $w \in C\left([0, T] ; L^2\left(\mathbb{R}^n\right)\right)$ and $\varepsilon<\delta$, we consider
$$
w_{\varepsilon}:(t, x) \in \Sigma_\delta \mapsto \chi_{\varepsilon}(x) \int_{(0, T)} \psi_{\varepsilon}(t-s)\left(\int_{\mathbb{R}^n} \phi_{\varepsilon}(x-y) w(s, y) \mathrm{d} y\right) \mathrm{d} s
$$
We have
\begin{align}\label{eq_A3}
  \lim _{\varepsilon \rightarrow 0}\left\|w(t, \centerdot)-w_{\varepsilon}(t, \centerdot)\right\|_{L^2\left(\mathbb{R}^n\right)}=0, \quad \forall t \in[\delta, T-\delta] .  
\end{align}
Furthermore, if $\left(\partial_{\mathrm{t}}\pm\Delta\right) w \in L^1\left((0, T) ; L^2\left(\mathbb{R}^n\right)\right)$, then
\begin{align}\label{eq_A4}
  \lim _{\varepsilon \rightarrow 0}\left\|\left( \partial_{\mathrm{t}} \pm \Delta\right) w-\left( \partial_{\mathrm{t}} \pm \Delta\right) w_{\varepsilon}\right\|_{L^1\left((\delta, T-\delta) ; L^2\left(\mathbb{R}^n\right)\right)}=0  
\end{align}
Let us give an explanation about \eqref{eq_A4}. For notational convenience, we write $\varphi_{\varepsilon}(t, x)=\psi_{\varepsilon}(t) \phi_{\varepsilon}(x)$ for $(t, x) \in \mathbb{R} \times \mathbb{R}^n$, and let $\tilde{w}$ denote the trivial extension of $w$
$$
\tilde{w}(t, \centerdot)= \begin{cases}w(t, \centerdot) & t \in[0, T], \\ 0 & t \notin[0, T] .\end{cases}
$$
Thus, we have that $w_{\varepsilon}(t, x)=\chi_{\varepsilon}(x)\left(\varphi_{\varepsilon} * \tilde{w}\right)(t, x)$ for all $(t, x) \in \Sigma_\delta$. In order to show that \eqref{eq_A4} holds, let us compute
$$
\left(\partial_{\mathrm{t}} \pm \Delta\right) w_{\varepsilon}=\chi_{\varepsilon}\left( \partial_{\mathrm{t}} \pm \Delta\right)\left(\varphi_{\varepsilon} * \tilde{w}\right) \pm 2 \nabla \chi_{\varepsilon} \cdot \nabla\left(\varphi_{\varepsilon} * \tilde{w}\right) \pm \Delta \chi_{\varepsilon}\left(\varphi_{\varepsilon} * \tilde{w}\right),~~ \text { in } \Sigma_\delta .
$$
Thus, \eqref{eq_A4} follows from
\begin{align}
 \lim _{\varepsilon \rightarrow 0}\left\|\left( \partial_{\mathrm{t}}\pm\Delta\right) w-\chi_{\varepsilon}\left( \partial_{\mathrm{t}}\pm\Delta\right)\left(\varphi_{\varepsilon} \ast \tilde{w}\right)\right\|_{L^1\left((\delta, T-\delta) ; L^2\left(\mathbb{R}^n\right)\right)}=0,  \label{eq_A5} 
\end{align}
and
\begin{align}\label{eq_A6} 
 &\lim _{\varepsilon \rightarrow 0}\Big(\left\|\nabla \chi_{\varepsilon} \cdot \nabla\left(\varphi_{\varepsilon} \ast \tilde{w}\right)\right\|_{L^1\left((\delta, T-\delta) ; L^2\left(\mathbb{R}^n\right)\right)}
 \cr&\qquad\quad+\left\|\Delta \chi_{\varepsilon}\left(\varphi_{\varepsilon} \ast \tilde{w}\right)\right\|_{L^1\left((\delta, T-\delta) ; L^2\left(\mathbb{R}^n\right)\right)}\Big)
 =0.   
\end{align}
Let us start by proving \eqref{eq_A5}. By the triangle inequality
\begin{align}\label{eq_A7}
&\|\left(\partial_{\mathrm{t}}\pm\Delta\right) w-\chi_{\varepsilon}\left( \partial_{\mathrm{t}}\pm\right.  \Delta)\left(\varphi_{\varepsilon} \ast \tilde{w}\right) \|_{L^1\left((\delta, T-\delta) ; L^2\left(\mathbb{R}^n\right)\right)} \cr
&\leq  \left\|\left( \partial_{\mathrm{t}}\pm\Delta\right) w-\varphi_{\varepsilon} \ast\left[\left( \partial_{\mathrm{t}}\pm\Delta\right) \tilde{w}\right]\right\|_{L^1\left((\delta, T-\delta) ; L^2\left(\mathbb{R}^n\right)\right)} \cr
&\qquad+\left\|\left(1-\chi_{\varepsilon}\right)\left( \partial_{\mathrm{t}}\pm\Delta\right) w\right\|_{L^1\left((\delta, T-\delta) ; L^2\left(\mathbb{R}^n\right)\right)}    
\end{align}
To write the first term on the right-hand side as above, we have used $\left\|\chi_{\varepsilon}\right\|_{L^{\infty}\left(\mathbb{R}^n\right)}=1$, and $\left( \partial_{\mathrm{t}}\pm\Delta\right)\left(\varphi_{\varepsilon} \ast \tilde{w}\right)=\varphi_{\varepsilon} \ast\left[\left( \partial_{\mathrm{t}}\pm\Delta\right) \tilde{w}\right]$. It is obvious that
$$
\lim _{\varepsilon \rightarrow 0}\left\|\left(1-\chi_{\varepsilon}\right)\left( \partial_{\mathrm{t}}\pm\Delta\right) w(t, \centerdot)\right\|_{L^2\left(\mathbb{R}^n\right)}=0,~~ \text { for almost every } t \in(\delta, T-\delta) .
$$
Hence, by the dominate convergence theorem,
$$
\lim _{\varepsilon \rightarrow 0}\left\|\left(1-\chi_{\varepsilon}\right)\left( \partial_{\mathrm{t}}\pm\Delta\right) w\right\|_{L^1\left((\delta, T-\delta) ; L^2\left(\mathbb{R}^n\right)\right)}=0 .
$$
To ensure that \eqref{eq_A5} holds, it remains to analyze the first term on the right-hand side of inequality \eqref{eq_A7}. Since $\varepsilon<\delta$, then we get
$$
\varphi_{\varepsilon} \ast\left[\left( \partial_{\mathrm{t}}\pm\Delta\right) \tilde{w}\right](t, x)=\int_{\Sigma} \varphi_{\varepsilon}(t-s, x-y)\left[\left( \partial_{\mathrm{t}}\pm\Delta\right) w\right](s, y)d(s, y) \quad \text {a.e.}(t, x) \in \Sigma_\delta .
$$
Then, by using standard arguments we obtain
$$
\lim _{\varepsilon \rightarrow 0}\left\|\left( \partial_{\mathrm{t}}\pm\Delta\right) w-\varphi_{\varepsilon} \ast\left[\left( \partial_{\mathrm{t}}\pm\Delta\right) \tilde{w}\right]\right\|_{L^1\left((\delta, T-\delta) ; L^2\left(\mathbb{R}^n\right)\right)}=0 .
$$
Consequently, \eqref{eq_A5} holds. We move now to prove \eqref{eq_A6}. We begin by analyzing the first term in \eqref{eq_A6}. Applying Young’s convolution inequality and using the fact that $ \nabla\left(\varphi_{\varepsilon} \ast \tilde{w}\right)(t,x)= \big(\psi_{\varepsilon} \ast \left( \nabla \phi_{\varepsilon} \ast \tilde{w}\right)(\cdot, x)\big) (t)$ we obtain 
\begin{align*}
& \left\|\nabla \chi_{\varepsilon} \cdot \nabla\left(\varphi_{\varepsilon} \ast \tilde{w}\right)\right\|_{L^1\left((\delta, T-\delta) ; L^2\left(\mathbb{R}^n\right)\right)} \\
 &\qquad \leq \|  \nabla \chi_{\varepsilon} \|_{L^\infty\left(\mathbb{R}^n\right)} \|  \nabla\left(\varphi_{\varepsilon} \ast \tilde{w}\right)\|_{L^1\left((\delta, T-\delta) ; L^2\left(\mathbb{R}^n\right)\right)}\\
 &\qquad\leq \|  \nabla \chi_{\varepsilon} \|_{L^\infty\left(\mathbb{R}^n\right)} \|  \psi_{\varepsilon} \|_{L^1\left(\mathbb{R}\right)} \|  \nabla \phi_{\varepsilon} \|_{L^1\left(\mathbb{R}^n\right)} \|  \tilde{w} \|_{L^1\left((\delta, T-\delta) ; L^2\left(\mathbb{R}^n\right)\right)}\\
  &\qquad \leq (T-2 \delta) \|  \nabla \chi_{\varepsilon} \|_{L^\infty\left(\mathbb{R}^n\right)} \|  \psi_{\varepsilon} \|_{L^1\left(\mathbb{R}\right)} \|  \nabla \phi_{\varepsilon} \|_{L^1\left(\mathbb{R}^n\right)} \|  \tilde{w} \|_{L^\infty\left((0, T) ; L^2\left(\mathbb{R}^n\right)\right)}.
\end{align*}
Since $\nabla \chi_{\varepsilon}(x)={\varepsilon}^2 \nabla \chi(\varepsilon^2x)$, $\|  \nabla \phi_{\varepsilon} \|_{L^1\left(\mathbb{R}^n\right)}={\varepsilon}^{-1} \|  \nabla \phi\|_{L^1\left(\mathbb{R}^n\right)}$ and $\|  \psi_{\varepsilon} \|_{L^1\left(\mathbb{R}\right)}=1$ we deduce
\begin{align*}
& \left\|\nabla \chi_{\varepsilon} \cdot \nabla\left(\varphi_{\varepsilon} \ast \tilde{w}\right)\right\|_{L^1\left((\delta, T-\delta) ; L^2\left(\mathbb{R}^n\right)\right)} \\
 &\qquad \leq (T-2 \delta) \varepsilon \|  \nabla \chi \|_{L^\infty\left(\mathbb{R}^n\right)}  \|  \nabla \phi \|_{L^1\left(\mathbb{R}^n\right)} \|  \tilde{w} \|_{L^\infty\left((0, T) ; L^2\left(\mathbb{R}^n\right)\right)}.
\end{align*}
Proceeding similarly for the second term in \eqref{eq_A6}, we derive
\begin{align*}
& \left\|\Delta \chi_{\varepsilon} \left(\varphi_{\varepsilon} \ast \tilde{w}\right)\right\|_{L^1\left((\delta, T-\delta) ; L^2\left(\mathbb{R}^n\right)\right)} \\
  &\qquad \leq (T-2 \delta) \|  \Delta \chi_{\varepsilon} \|_{L^\infty\left(\mathbb{R}^n\right)} \|  \psi_{\varepsilon} \|_{L^1\left(\mathbb{R}\right)} \| \phi_{\varepsilon} \|_{L^1\left(\mathbb{R}^n\right)} \|  \tilde{w} \|_{L^\infty\left((0, T) ; L^2\left(\mathbb{R}^n\right)\right)}\\
 &\qquad \leq (T-2 \delta) {\varepsilon}^4\|  \Delta \chi \|_{L^\infty\left(\mathbb{R}^n\right)} \|  \tilde{w} \|_{L^\infty\left((0, T) ; L^2\left(\mathbb{R}^n\right)\right)}.
\end{align*}
Taking the limit, as $\varepsilon $ tends to $ 0$, in both inequalities yields \eqref{eq_A6}. Consequently, the desired result \eqref{eq_A4} follows. We now apply these results for the corresponding functions $u_{\varepsilon}$ and $v_{\varepsilon}$. A direct application of the standard integration-by-parts rules gives
\[ \int_{\Sigma_\delta}[( \partial_{\mathrm{t}}-\Delta) u_{\varepsilon}{v_{\varepsilon}}+u_{\varepsilon} ( \partial_{\mathrm{t}}+\Delta) v_{\varepsilon}]= \int_{\mathbb{R}^n}[u_{\varepsilon}(T-\delta,\centerdot) {v_{\varepsilon}(T-\delta, \centerdot)}+u_{\varepsilon}(\delta, \centerdot) {v_{\varepsilon}(\delta, \centerdot)}] . \]
After using \eqref{eq_A3} and \eqref{eq_A4} we can conclude, by the dominate convergence theorem, that
$$
\int_{\Sigma_\delta}\left[\left( \partial_{\mathrm{t}}-\Delta\right) u v+u {\left( \partial_{\mathrm{t}}+\Delta\right) v}\right]= \int_{\mathbb{R}^n}\left[u(T-\delta, \centerdot) {v\left(T-\delta{, \centerdot}\right)}+u(\delta, \centerdot) {v(\delta, \centerdot)}\right] .
$$
The limits \eqref{eq_A2} and
$$
\int_{\mathbb{R}^n}[u(T, \centerdot) {v(T, \centerdot)}-u(0, \centerdot) {v(0, \centerdot)}]=\lim _{\delta \rightarrow 0} \int_{\mathbb{R}^n}[u(T-\delta, \centerdot) {v(T-\delta, \centerdot)}+u(\delta, \centerdot) {v(\delta, \centerdot})].
$$
yield our desired result.
\end{proof}

\sloppy
\begin{acknowledgements}
E.A. and P.C. were supported by the grant PID2024-156267NB-I00 (funded by 
MICIU/AEI/10.13039/501100011033 and FEDER, UE), by BCAM-BERC 2026-2029
(funded by the Basque Government) and Severo Ochoa CEX2021-001142-S.
E.A. is also supported by TrafoSPINN funded by IKUR HPC-IA (Eusko Jaurlaritza)
N.A. was supported by a mobility scholarship funded by the University of Tunis El Manar during the initial stages of this work. Co-funded by the European Union (ERC, SAMPDE, 101041040). Views and opinions expressed are however those of the authors only and do not necessarily reflect those of the European Union or the European Research Council. Neither the European Union nor the granting authority can be held responsible for them. P.C. is also funded by Ikerbasque, the Basque Foundation for Science.
\end{acknowledgements}

\bibliography{references}{}

\begin{thebibliography}{10}

\bibitem{zbMATH07335388}
Mourad Bellassoued and Oumaima Ben~Fraj.
\newblock Stably determining time-dependent convection-diffusion coefficients
  from a partial {Dirichlet}-to-{Neumann} map.
\newblock {\em Inverse Probl.}, 37(4):35, 2021.
\newblock Id/No 045011.

\bibitem{zbMATH07173395}
Mourad Bellassoued and Imen Rassas.
\newblock Stability estimate for an inverse problem of the convection-diffusion
  equation.
\newblock {\em J. Inverse Ill-Posed Probl.}, 28(1):71--92, 2020.

\bibitem{arXiv:2607.13778}
Liam Buisson.
\newblock Recovery of coefficients for a convection-diffusion equation from
  partial data.
\newblock Preprint, {arXiv}:2607.13778 [math.{AP}] (2026), 2026.

\bibitem{zbMATH08122191}
Manuel Ca{\~n}izares, Pedro Caro, Ioannis Parissis, and Thanasis Zacharopoulos.
\newblock The initial-to-final-state inverse problem with time-independent
  potentials.
\newblock {\em Inverse Probl.}, 41(11):19, 2025.
\newblock Id/No 115001.

\bibitem{arXiv:2602.12122}
Manuel Ca{\~n}izares, Pedro Caro, Ioannis Parissis, and Thanasis Zacharopoulos.
\newblock The initial-to-final-state inverse problem with critically-singular
  potentials.
\newblock Preprint, {arXiv}:2602.12122 [math.{AP}] (2026), 2026.

\bibitem{zbMATH01578850}
B.~Canuto and O.~Kavian.
\newblock Determining coefficients in a class of heat equations via boundary
  measurements.
\newblock {\em SIAM J. Math. Anal.}, 32(5):963--986, 2001.

\bibitem{arXiv:1812.08495}
Pedro Caro and Yavar Kian.
\newblock Determination of convection terms and quasi-linearities appearing in
  diffusion equations.
\newblock Preprint, {arXiv}:1812.08495 [math.{AP}] (2018), 2018.

\bibitem{zbMATH07801151}
Pedro Caro and Alberto Ruiz.
\newblock An inverse problem for data-driven prediction in quantum mechanics.
\newblock {\em J. Math. Phys.}, 65(1):28, 2024.
\newblock Id/No 011505.

\bibitem{arXiv:2512.04796}
Pedro Caro and Alberto Ruiz.
\newblock The initial-to-final-state inverse problem with unbounded potentials
  and {Strichartz} estimates.
\newblock Preprint, {arXiv}:2512.04796 [math.{AP}] (2025), 2025.

\bibitem{zbMATH06869653}
Mourad Choulli and Yavar Kian.
\newblock Logarithmic stability in determining the time-dependent zero order
  coefficient in a parabolic equation from a partial {Dirichlet}-to-{Neumann}
  map. application to the determination of a nonlinear term.
\newblock {\em J. Math. Pures Appl. (9)}, 114:235--261, 2018.

\bibitem{zbMATH07802400}
Ali Feizmohammadi.
\newblock An inverse boundary value problem for isotropic nonautonomous heat
  flows.
\newblock {\em Math. Ann.}, 388(2):1569--1607, 2024.

\bibitem{zbMATH07543698}
Yavar Kian.
\newblock Simultaneous determination of different class of parameters for a
  diffusion equation from a single measurement.
\newblock {\em Inverse Probl.}, 38(7):29, 2022.
\newblock Id/No 075008.

\bibitem{zbMATH08109740}
Rohit~Kumar Mishra, Anamika Purohit, and Manmohan Vashisth.
\newblock Inverse problem for a time-dependent convection-diffusion equation in
  admissible geometries.
\newblock {\em Res. Math. Sci.}, 12(4):30, 2025.
\newblock Id/No 75.

\bibitem{zbMATH08020896}
Anamika Purohit.
\newblock Determining time-dependent convection and density terms in a
  convection-diffusion equation using partial data.
\newblock {\em Commun. Anal. Comput.}, 3:69--90, 2025.

\bibitem{zbMATH07829439}
Imen Rassas.
\newblock H{\"o}lder stability estimates in determining the time-dependent
  coefficients of the heat equation from the {Cauchy} data set.
\newblock {\em J. Inverse Ill-Posed Probl.}, 32(2):183--198, 2024.

\bibitem{zbMATH07170240}
Suman~Kumar Sahoo and Manmohan Vashisth.
\newblock A partial data inverse problem for the convection-diffusion equation.
\newblock {\em Inverse Probl. Imaging}, 14(1):53--75, 2020.

\bibitem{zbMATH07597186}
Soumen Senapati and Manmohan Vashisth.
\newblock Stability estimate for a partial data inverse problem for the
  convection-diffusion equation.
\newblock {\em Evol. Equ. Control Theory}, 11(5):1681--1699, 2022.

\end{thebibliography}
\bibliographystyle{plain}

\end{document}